\documentclass[a4paper,10pt]{article}
\usepackage[utf8]{inputenc}

\usepackage{amsfonts}
\usepackage{color,enumerate,wrapfig,amssymb}
\usepackage{amsmath}
\usepackage{amsthm}
\usepackage{esint}
\usepackage{mathrsfs}
\usepackage{upgreek}
\usepackage{tipa}
\usepackage{tikz}
\usepackage{geometry}

\theoremstyle{plain}
\newtheorem{theorem}{Theorem}
\newtheorem{lemma}[theorem]{Lemma}
\newtheorem{proposition}[theorem]{Proposition}
\newtheorem{corollary}[theorem]{Corollary}
\newtheorem{remark}[theorem]{Remark}
\newtheorem{example}[theorem]{Example}
\newtheorem{definition}[theorem]{Definition}

\numberwithin{equation}{section}

\numberwithin{theorem}{section}

\newcommand{\dist}{dist}
\newcommand{\sgn}{sgn}

\title{Analysis of blow-ups for the double obstacle problem in dimension two}

\author{Gohar Aleksanyan
\thanks{ {Current affiliation: Department of Mathematics, University of Duisburg Essen, Thea-Leymann-Strasse 9, 45127 Essen, Germany. E-mail: gohar.aleksanyan@uni-due.de} } }

\begin{document}

\maketitle

\begin{abstract}
In this article we study a normalised double obstacle problem with polynomial obstacles $ p^1\leq p^2$ under 
the assumption that $ p^1(x)=p^2(x)$ iff $ x=0$.
In dimension two  we give a complete characterisation of blow-up solutions 
depending on the coefficients
of the polynomials $p^1, p^2$. In particular, we see that there exists a new type of blow-ups, that
we call double-cone solutions since the coincidence sets $\{u=p^1\}$ and $\{u=p^2\}$ are  cones
with a common vertex. 

\par
We prove the uniqueness of blow-up limits, and analyse the regularity of the free boundary in dimension two.
In particular we show that if the solution to the double obstacle problem has a double-cone blow-up limit at the origin, then locally the free boundary consists of four $C^{1,\gamma}$-curves, meeting at the origin.

\par
In the end we give an example of a three-dimensional double-cone solution.

\end{abstract}

\tableofcontents

\section{Introduction}

\par
Let $ \Omega  $ be a bounded open set in $ \mathbb{R}^n$ with 
smooth boundary. 
The solution to the double obstacle problem in $\Omega$ is the minimiser of the functional
\begin{equation*}
 J(v)= \int_\Omega \lvert \nabla v(x) \rvert^2 dx
\end{equation*}
over  functions $ v \in W^{1,2}(\Omega)$, $ \psi^1 \leq v \leq \psi^2$, satisfying the boundary condition $v=g $
on $ \partial \Omega$. For the problem to be well defined we assume that $ \psi^1 \leq \psi^2 $ in $\Omega$ and
$ \psi ^1 \leq g \leq \psi^2 $ on 
$ \partial \Omega$.
The functions $\psi^1$ and  $\psi^2$
 are called respectively the lower and the upper obstacles. 

\par 
 If $ \psi^1 < \psi^2 $ then the problem reduces locally to a single obstacle problem. Therefore we are interested
in the case when 
\begin{equation} \label{Lambda}
{\varLambda}:=\{x \in \Omega:\psi^1(x)=\psi^2(x) \} \neq \emptyset.
\end{equation}

\par
It is well known  that the solution to the double obstacle problem  satisfies the 
following inequalities
\begin{equation} \label{dopb}
\begin{aligned}
 \psi^1 \leq  u \leq \psi^2, ~~\Delta u \geq 0 ~\textrm{ if } u > \psi^1 \textrm{ and } ~
 \Delta u  \leq 0 ~\textrm{ if } u < \psi^2.
\end{aligned}
\end{equation}
It has been shown that the solution to the double obstacle problem is locally $ C^{1,1}$ under the assumption
$ \psi^i \in C^2(\Omega)$, see 
for instance \cite{ASW12,FS15}. 
Therefore we may rewrite \eqref{dopb} as
\begin{equation} \label{equat}
\begin{aligned}
\psi^1 \leq u \leq \psi^2 ~\textrm{ and }~
 \Delta u  
 =\Delta \psi^1 \chi_{\{u=\psi^1\}} + \Delta \psi^2 \chi_{\{u=\psi^2\}}- \Delta \psi^1 \chi_{ \{\psi^1= \psi^2\}} ~
 ~a.e.,
 \end{aligned}
\end{equation}
where $ \chi_A$ is the characteristic function of a set $A \subset \mathbb{R}^n$.

\par
Let us introduce some notations 
that will be used throughout.
Denote by 
\begin{equation} \label{omega}
 \Omega_1:=\{ u >\psi^1 \}, ~  ~ \Omega_2:=\{ u < \psi^2 \}, \textrm{ and } \Omega_{12}:= \Omega_1 \cap \Omega_2
\end{equation}
then $ \Omega= \Omega_1 \cup \Omega_2 \cup \varLambda$, where $ \varLambda$ is given by \eqref{Lambda}.
Let us observe that $ u$ is a harmonic function in $ \Omega_{12}$, which we call
 the noncoincidence set. 
Define the free boundary for the double obstacle problem
\begin{equation} \label{Gamma}
 \varGamma:= \partial \Omega_{12} \cap \Omega \subset \varGamma_1 \cup \varGamma_2, 
 \textrm{ where } ~\varGamma_i:= \partial \Omega_i \cap \Omega, ~ i=1,2.
 \end{equation}
 Let  $ x_0 \in  \varGamma$ be a  free boundary point, if $ x_0 \in \varGamma_1 \setminus \varGamma_2$, or if 
 $ x_0 \in \varGamma_2 \setminus \varGamma_1$, then locally we are in the setting of the classical obstacle problem.
 In this case the known regularity theory for the classical obstacle problem (see \cite{Caf98})
 can be applied to analyse the free boundary $ \varGamma$ in a neighbourhood of $ x_0$. Hence we are more curious about the behaviour of the free boundary
 at the points $ x_0 \in \varGamma_1 \cap \varGamma_2= \partial \varLambda$.
In this article we focus on the case when $ x_0 \in \varGamma_1 \cap \varGamma_2$ is an
isolated point of $ \varLambda$. The work is inspired by the following example of a homogeneous of degree two solution
in $ \mathbb{R}^2$,
\begin{equation} \label{john}
 u_0(x ) =x_1^2 \sgn(x_1)  + x_2^2 \sgn(x_2),
\end{equation}
where the obstacles $ p^1(x)=-p^2(x)=-x_1^2-x_2^2$, and $ \varLambda= \{0\}$. 
Example \eqref{john} has also been considered in \cite{GA}, when investigating the optimal regularity in the optimal
switching problem. The optimal switching problem and the double obstacle problem are related, and we see
that in both cases the solution shows a new type of behaviour at isolated points of $ \varLambda$.
The function $ u_0$ is a motivational example
for double-cone solutions, see Definition \ref{doublecone}.

\par
Before proceeding to the results obtained in the paper, let us mention that in the recent paper 
\cite{SPL17} the regularity of the free boundary for the double obstacle problem is studied by relaxing one of the obstacles. 
Under a thickness assumption, the authors in \cite{SPL17} show that the possible blow-ups are halfspace solutions, and prove 
 the $C^1$-regularity 
of the free boundary.

\subsection{Summary of the results}

We consider a normalised double obstacle problem in dimension $ n=2$, with polynomial obstacles $ p^1\leq p^2$;
\begin{equation} \label{normalizeddopb}
 \Delta u = \lambda_1 \chi_{\{ u= p^1\}}+\lambda_2 \chi_{\{ u= p^2 \}},
 \end{equation}
where $ \lambda_1 = \Delta p^1 <0$ and  $\lambda_2 =\Delta p^2 >0$ are constants. Furthermore, we assume that 
$p^ 1$ and  $ p^ 2 $ meet at a single point, i.e. $p^1(x)=p^ 2(x)$ iff $ x=x_0$.

\par
 Without loss of generality, we may assume that $ x_0=0$, and the polynomials $p^i$ are of the form
\begin{equation} \label{polynomialobstacles}
 p^1(x)=a_1 x_1^2+c_1 x_2^2 ~\textrm{ and }~ p^2(x)=a_2x_1^2+c_2x_2^2,
\end{equation}
where
\begin{equation*}
 a_1 +c_1 <0, ~a_2 +c_2 >0, ~\textrm{and } a_1<a_2, ~c_1<c_2.
\end{equation*}

\par
The paper is structured as follows.
In Section 2 we study the normalised double obstacle problem \eqref{normalizeddopb}.
We show that the blow-ups of the solution to the normalised double obstacle are homogeneous 
of degree two functions via Weiss' monotonicity formula.

\par
Knowing that the blow-up solutions are homogeneous of degree two functions, in \mbox{Section 3} we make a complete 
characterisation of possible blow-ups in dimension $ n=2$ (Theorem  \ref{theorem1} and Theorem \ref{theoremhalfspace}).  In particular we see that there exist 
 homogeneous of degree 
two global solutions of a new type. We call these solutions double-cone solutions, since the coincidence sets $\{u=p^1\} $ and $ \{u=p^2\} $ are cones 
with a common vertex at the origin. We show that there exist double-cone solutions if and only if the
following polynomial
\begin{equation}
 P=P(x_1, x_2) \equiv p^1(x_1, x_2) + p^2(x_2, x_1) =(a_1+c_2)x_1^2+(a_2+c_1) x_2^2
\end{equation}
has no sign (Corollary \ref{core}). 
 The existence of halfspace solutions corresponding to $ p^i$ is also studied (Theorem \ref{theoremhalfspace} and Corollary \ref{corehalfspace}). In particular, we see that a halfspace solution to the obstacle problem with 
$ p^1$ (or $ -p^2$) is not necessarily a halfspace solution to the double obstacle problem with obstacles $p^1 \leq p^2$. 

\par
Given obstacles \eqref{polynomialobstacles}, there are three different cases, depending on the coefficients of $p^i$, that describe the possible 
blow-up solutions for the double obstacle problem with $p^1 \leq p^2$.
We show the uniqueness of blow-up limits and analyse the behaviour of the free boundary in these three cases separately. 

\par
\textbf{\textit {Case 1:}}
 If $ P\equiv 0$, there are infinitely many double-cone solutions. This is perhaps the most interesting case,  it is studied in Section 4.
This case can be reduced to the  double obstacle problem with obstacles
$ p^1(x) = -x_1^2-x_2^2$, $ p^2(x)= x_1^2+ x_2^2 $. By using a version of a flatness improvement argument, we show 
that if the solution is close to a double-cone solution in $ B_1$, then the blow-up at the origin is unique. 
Furthermore, employing the known regularity theory for the free boundary in the classical problem, 
we derive that the free boundary $ \varGamma$ for the double obstacle problem is a union of four 
$ C^{1,\gamma}$-graphs meeting at the origin, see Theorem \ref{theorem2}. Neither $ \varGamma_1$ nor $ \varGamma_2$ is flat at the origin,
and they meet at right angles, see Figure 4.1.

\par
In this case there are infinitely many rotationally invariant halfspace solutions $u$ corresponding to $ p^1 $ (or $ p^2$),
and the set $ \{u=p^2\}$ (or $\{u=p^1\}$) is a halfline.  Via  a flatness improvement argument, we show that if the solution to the double is close to a halfspace solution corresponding to  $ p^1$, then $ \varGamma_1$ is a $ C ^{1,\gamma}$-curve in a neighbourhood of the origin. The proof of the last statement is the same in all three cases.

\textbf{\textit {Case 2:}} If $ P $ changes the sign, i. e. $D^2 P $ has two eigenvalues with opposite sign, then there are only four double-cone solutions, and it follows that the blow-up at the origin is unique (Theorem \ref{case2uniqueness}). Furthermore, we show that if the solution to the double obstacle problem has a double-cone blow-up limit, then locally  the free boundary consists of four $C^{1,\gamma}$-curves, meeting at the origin.

\par
In \textit{Case 2} there are infinitely many halfspace solutions corresponding to $p^i$, which are not rotationally invariant on the plane, i.e. the rotation of $ \varGamma_i$ can be 
performed only inside a fixed cone. Hence not every direction on the  plane gives a
halfspace solution.

\textbf{\textit {Case 3:}} The polynomial $ P$ has a sign. There are no double-cone solutions in this case.
We show that if $ P\geq 0$, then there are infinitely many rotationally invariant halfspace solutions corresponding to $ p^1$. Furthermore, if  $D^2 P$ is a positive definite matrix (both eigenvalues are positive) then there are no halfspace solutions corresponding to the upper obstacle $p^2$. Similarly,  if $ P(x)\leq 0$ there are infinitely many halfspace solutions 
corresponding to $p^2$, and if $D^2P <0$, there are no halfspace solutions corresponding to the lower obstacle $p^1$.
Hence the solution chooses the obstacle having lower curvature.

\par
Let us also mention an important property of the double obstacle problem, following from our discussion of \textit{Cases 1, 2} and \textit{3}.
Let $ \varepsilon$ be an arbitrary number, $ \lvert \varepsilon \rvert <<1$.
Then for polynomials 
$ p^1(x) = -x_1^2-x_2^2$, $ p^2(x)= x_1^2+ x_1^2 $ there exist infinitely many double-cone solutions. While 
when we look at the double obstacle problem with $ {p}^1 = -x_1^2 -x_2^2 $ and $ \tilde{p}^2 =( 1-\varepsilon)x_1^2 +
(1+\varepsilon)x_2^2$ there are only four double-cone solutions, and for $ {p}^1 = -x_1^2 -x_2^2 $ and 
$ \bar{p}^2 =( 1+\varepsilon)x_1^2 +
(1+\varepsilon)x_2^2$ there are none. 
This property of the double obstacle problem is quite surprising and unexpected. It reveals the 
instability of the solutions
in the sense that changing the 
obstacles slightly, may change the solution and the free boundary significantly.

\par
It is an interesting question to investigate double-cone solutions also in higher dimensions.
In the end of the paper we give an example of a three-dimensional double-cone solution. The complete  analysis of blow-up solutions for the double obstacle problem in $\mathbb{R}^3$ we leave for a future publication.

\subsection*{Acknowledgements}
The paper is a part of my doctoral thesis, written at KTH, Royal Institute of Technology in Stockholm. I am grateful to my advisor, Prof. Dr. John Andersson,
 for his guidance and support throughout the project.

\par
I would also like to thank  Dr. Erik Lindgren and Prof. Dr. Henrik Shahgholian 
for reading a preliminary version of the 
manuscript and for their valuable feedback.

\bigskip

\section{Weiss' energy functional for the double obstacle problem}

In this section we study the behaviour of the solutions locally at free boundary points via Weiss' monotonicity
formula.

\par
Let $ u $ be a solution to  the double problem in $\Omega$, with obstacles 
\begin{equation} \label{psi}
 \psi^1 \leq \psi^2, ~\psi^1, \psi^2 \in C^{2}(\Omega), ~ {\varLambda}=\{ \psi^1=\psi^2 \} \neq \emptyset.
\end{equation}

\par
Fix any $ x_0 \in \varGamma \cap \partial \varLambda$ and assume that $ B_1(x_0) \subset \Omega$. 
Denote by
\begin{equation*}
 u_{r,x_0}:=\frac{u(rx + x_0)-u(x_0)-r \nabla u(x_0)\cdot x}{r^2}, \textrm{ for all } 0< r<1,
 \textmd{ and } x_0 \in \varGamma.
\end{equation*}
Without loss of generality, assume that $ x_0=0$ and $ B_1\subset \Omega$.
Furthermore, by subtracting a first order polynomial from $u$, we may assume that
$ u(0)= \lvert \nabla u(0)\rvert=0$.
Recalling that $ u\in C^{1,1}_{loc}$, we obtain
$ \psi^1(0)=\psi^2(0)=u(0)=0$ and
$ \lvert \nabla \psi^1(0)\rvert=\lvert \nabla \psi^2(0)\rvert=\lvert \nabla u(0) \rvert=0$.
Denote by
\begin{equation}\label{urx0}
 u_r(x):=u_{r,0}= \frac{u(rx)}{r^2}.
\end{equation}

\par
It follows from equation \eqref{dopb} and assumption \eqref{psi}, that $\lambda_1= \Delta \psi^1(0)\leq 0$
and $ \lambda_2 =\Delta \psi^2(0) \geq 0$. In particular, if $ 0 \in \partial \varLambda^\circ $, then 
$ \lambda_1= \lambda_2 =0$.

\begin{lemma} \label{lemmaweiss}
Consider the following normalised double obstacle problem
\begin{equation} \label{u}
 \Delta u = \lambda_1 \chi_{\{u=\psi^1\}} + \lambda_2 \chi_{\{u=\psi^2\}}, \textrm{ in } B_1
\end{equation}
where $ \psi^i \in C ^2(B_1)$, and assume that 
\begin{equation} \label{signedlambda}
 \lambda_1:= \Delta  \psi^1 \leq 0 \textrm{ and } \lambda_2:= \Delta  \psi^2 \geq 0 \textrm{ are constants}.
\end{equation}
Define Weiss' energy functional for the function $u $ and $ 0<r\leq 1$ at the origin as follows
\begin{equation} \label{weiss}
\begin{aligned}
 W( u,r, 0):=\frac{1}{r^{n+2}} \int_{B_r}  \lvert \nabla u \rvert^2 dx 
 -\frac{2}{r^{n+3}} \int_{\partial B_r} u^2 d\mathcal{ H}^{n-1} \\
 + \frac{1}{r^{n+2}}\int_{B_r} 2 \lambda_1 u \chi_{\{ u = \psi^1\}}+ 2 \lambda_2 u \chi_{ \{u=\psi^2\}} dx.
 \end{aligned}
\end{equation}
Then
\begin{equation} \label{weiss0}
   \frac{d}{dr}  W( u,r,0)=2r \int _{\partial B_1} \left(  \frac{du_r}{dr} \right)^2 d \mathcal{H}^{n-1} \geq 0.
\end{equation}

\end{lemma}

\begin{proof}

\par
After a change of variable in \eqref{weiss} we obtain the following scaling property for Weiss' energy functional
\begin{equation} \label{scaling}
\begin{aligned}
 W( u,r,0)=W( u_r,1,0)= 
  \int_{B_1}  \lvert \nabla u_r \rvert^2 dx  -2 \int_{\partial B_1} u_r^2 d \mathcal{ H}^{n-1} \\
 + \int_{B_1} 2 \lambda_1 u_r \chi_{\{ u_r = \psi^1_r\}}+2  \lambda_2 u_r \chi_{ \{u_r= \psi^2_r\}} dx.
\end{aligned}
\end{equation}
Hence
\begin{align*}
 \frac{d}{dr}  W( u,r,0)= \frac{d}{dr}W( u_r,1,0)=\int_{B_1}  \frac{d}{dr}
 \lvert \nabla u_r \rvert^2 dx  -2\int_{\partial B_1} \frac{du_r^2}{dr}  d \mathcal{ H}^{n-1} \\
 +  2 \int_{B_1} \left( \lambda_1  \chi_{\{ u_r = \psi^1_r\}}+ 
  \lambda_2  \chi_{ \{u_r=\psi^2_r\}} \right) \frac{du_r}{dr}  dx 
 =2 \int_{B_1}   \nabla \frac{d u_r}{dr} \nabla u_r  dx \\
 -4 \int _{\partial B_1} \frac{du_r}{dr} u_r   d \mathcal{H}^{n-1} 
 +2  \int_{B_1} \left(  \lambda_1  \chi_{\{ u_r = \psi^1_r\}}+ 
  \lambda_2  \chi_{ \{u_r=\psi^2_r\}} \right)\frac{du_r}{dr} dx .
\end{align*}
By Green's formula
\begin{equation*}
 \int_{B_1}  \nabla u_r \nabla \frac{d u_r}{dr}  dx= - \int_{B_1}\frac{du_r}{dr} \Delta u_r  dx +
\int _{\partial B_1} \frac{du_r}{dr}\frac{ \partial u_r}{\partial \nu}  d\mathcal{ H}^{n-1}.
\end{equation*}
Therefore
\begin{equation} \label{dr}
\begin{aligned}
  \frac{d}{dr}  W( u,r,0)=
   2 \int _{\partial B_1}\frac{du_r}{dr} \left( \frac{\partial u_r }{ \partial \nu}-2u_r \right)   d \mathcal{H}^{n-1}\\
  + 2 \int_{B_1}  \frac{du_r}{dr} \left( -\Delta u_r +\lambda_1 \chi_{\{ u_r = \psi^1_r\}}+ 
\lambda_2 \chi_{ \{u_r= \psi^2_r \} } \right)  dx.
 \end{aligned}
\end{equation}
Since $ u$ solves \eqref{u}, equation \ref{dr} can be abbreviated to 
\begin{equation} \label{almostweiss}
  \frac{d}{dr}  W( u,r,0)=
   2 \int _{\partial B_1}\frac{du_r}{dr} \left( \frac{\partial u_r }{ \partial \nu}-2u_r \right)  
   d \mathcal{H}^{n-1}.
\end{equation}
Let us observe that 
\begin{equation} \label{hom}
\begin{aligned}
 \int _{\partial B_1}\frac{du_r}{dr} \left( \frac{\partial u_r }{ \partial \nu}-2u_r \right)   d\mathcal{ H}^{n-1}=
 \int_{\partial B_1}\frac{du_r}{dr} \left(  x \cdot \nabla u_r-2u_r \right) d \mathcal{H}^{n-1} \\
  = \int _{\partial B_1} r \left(  \frac{du_r}{dr} \right)^2 d \mathcal{ H}^{n-1}.
\end{aligned}
\end{equation}

Equations \eqref{hom} and  \eqref{almostweiss} together imply the desired identity, \eqref{weiss0}.

\end{proof}

\bigskip

\section{Characterisation of blow-ups in $ \mathbb{R}^2 $}

Given second degree polynomials $ p^1 \leq p^2$, satisfying
\eqref{signedlambda}, let $ u$ be the solution 
to the normalised double obstacle problem \eqref{u} with $ p^1, p^ 2$. Let
$ 0\in \varGamma_1 \cap \varGamma_2 $ be a
free boundary point.
 By subtracting a first order polynomial from $ p^1$, $ p^2$ and $ u$, and recalling that
 $ u\in  C^{1,1}$, 
we may assume
\begin{equation} \label{quadratic}
 u(0)=p^1(0)=p^2(0)=0\textrm{ and }
 \lvert \nabla u(0) \rvert=\lvert \nabla p^1(0)\rvert=\lvert \nabla p^2(0)\rvert=0.
 \end{equation}
Hence $ p^1 $ and $ p^2 $ are homogeneous second degree polynomials.
 
\par
It follows from Lemma \ref{lemmaweiss} that $W(u, r, 0)$ is a nondecreasing 
absolutely continuous function in the interval $ (0,1)$. Hence there exists 
\begin{equation} \label{abc}
\lim_{r\rightarrow 0} W(u, r, 0) := W (u, 0+, 0). 
\end{equation}

\par
Since $ u \in C^{1,1}_{loc}$, we may conclude that
$ \Arrowvert u_r \rVert_{C^{1,1}} $ is uniformly bounded for small $r>0$. Therefore through a subsequence
$ u_r$ converges in $C^{1,\alpha}(B_1)$.
Let $ u_0$ be a blow-up of $ u $ at
the origin; 
\begin{equation} \label{subseq}
  \frac{u(r_j  x)}{r_j ^2} \rightarrow u_0  \textrm{ in } C^{1,\alpha}(B_1),
\end{equation}
for a sequence $ r_j \rightarrow 0+$, as $ j\rightarrow\infty $.
Then \eqref{subseq} implies that for any fixed $  0<r<1$
\begin{equation*}
 W(u_0, r, 0)= \lim_{j\rightarrow \infty} W(u_{r_j}, r, 0 ) \stackrel{ \eqref{scaling}}{=}
 \lim_{j\rightarrow \infty} W(u, rr_j, 0) \stackrel{\eqref{abc}}{=} W(u, 0+, 0).
\end{equation*}
Thus $  W(u_0, r, 0)$ has a constant value for all $ 0<r<1$, and  $ \frac{d}{dr}  W(u_0, r, 0) =0$.
Note that $ u_0$ is a global solution, i.e. solution in $ \mathbb{R}^n$ to the double obstacle problem with the 
same obstacles, $ p^1$ and $ p^2$.
Applying Lemma \ref{lemmaweiss} for the solution $ u_0$, we may conclude from 
\eqref{weiss0},  that 
\begin{equation*}
\frac{d }{dr} \left( \frac{u_0(rx)}{r^2} \right)=0, \textrm{ for any }  r>0.
\end{equation*}
Hence
$ u_0$ is a homogeneous of degree two function, which means that
\begin{align*}
 u_0(x)= \frac{u_0(rx)}{r^2},~ \textrm{ for any } ~x \in \mathbb{R}^n \textrm{ and } ~r>0.
\end{align*}
It follows that $ \Delta u_0(rx) =\Delta u_0(x) $, for any $x \in \mathbb{R}^n $  and $ ~r>0$. In other words
$ \Delta u_0 $ is identically constant on the lines passing through the origin, and therefore the free
boundary of $u_0$ is lying on straight lines passing through the origin.

\subsection{Examples}

\par
In this section we study motivational examples of homogeneous of degree two global solutions in $ \mathbb{R}^2$,
assuming that $ \varLambda= \{0\}$.

\par
It is well known that the (single) obstacle problem has two types of blow-ups; polynomial and halfspace solutions. 
The first obvious question is the following; if or when the  halfspace solutions to the obstacle problem are also solutions to the double obstacle problem, 
and  second, if the double obstacle problem has any other type of blow-ups which the obstacle problem does not.

\par
By using comparison principles, it is easy to see  that if  $ u $ is a polynomial solution to the double obstacle problem \eqref{u},
then $ u \equiv p^1$, $ u\equiv p^2$ or otherwise  $ u$ is a 
 homogeneous of degree two harmonic polynomial in $ \mathbb{R}^2$, such that $ p^1 \leq  u \leq p^2 $. 
 
 \par
Let us also recall the definition of a halfspace  solution in $\mathbb{R}^n$ (or halfplane in dim $ n=2$).
\begin{definition} \label{halfspace}
Let $ p^1\leq p^2$ be given  homogeneous degree two polynomials in $ \mathbb{R}^n$, satisfying $ \lambda_1=
\Delta p^1 <0$ and  $ \lambda_2=\Delta p^2>0$.
We say that $ u $ is a halfspace solution to the double obstacle problem corresponding to the lower obstacle $ p^1$, if $ u \leq p^2$ in $\mathbb{R}^n$,
and $ u-p^1 =- \frac{\lambda_1}{2}(x \cdot e)_+^2$, where $ e $ is a unit vector in $\mathbb{R}^n$.
Similarly, $ u $ is a halfspace solution corresponding to the upper obstacle $ p^2$, if $ u \geq p^1$ in $\mathbb{R}^n$, and 
$ p^2-u = \frac{\lambda_2}{2}(x \cdot e)_+^2$.
\end{definition}

\par
It follows from Definition \ref{halfspace} that if $ u $ is a halfspace solution corresponding to $ p^1$, then
$ \Delta u =\lambda_1 \chi_{\{(x\cdot e )<0\}}$, and  $ u < p^2$ a.e.. Similarly, 
if $ u $ is a halfspace solution corresponding to $ p^2$, then
$ \Delta u =\lambda_2 \chi_{\{(x\cdot e) <0\}}$, and  $ u > p^1 $ a.e.. 

\par
 In the following examples instead of our usual notation $ x=(x_1, x_2 ) \in \mathbb{R}^2$, 
the pair  $(x, y )$ represents a point in $\mathbb{R}^2$. It is done to make the pictures clearer, and we hope it
will not be confusing later on.

\begin{example} \label{ex1}
Let us study some explicit homogeneous degree two solutions to the double obstacle problem in $ \mathbb{R}^2$,
with fixed obstacles
$ p^1(x,y)= -x^2-y^2 , ~p^2(x,y)=x^2+y^2$.
\end{example}
Observe that $ u_0 = -x^2 + \sgn (y) y^2$ and $ u_0=\sgn(x)x^2 + y^2 $ are halfspace solutions corresponding to $p^1 =-x^2 -y^2$ and to $ p^2=x^2+y^2$
respectively;

\begin{samepage}
\begin{center}
 \begin{tikzpicture}[scale=1.1]
\draw[ultra thin,->] (-2.5,0) -- (2.5,0) node[anchor=west] {$x$} ;
\draw[ultra thin, ->] (0,-2.2) -- (0,2.2) node[anchor=south] {$y$} ;
\draw[thick,red]  (0,0) -- (0,2)node[right]{$\varGamma_2$} ;
\draw[thick,red]   (0,0) -- (0,2);
\draw[thick,blue]  (0,0) -- (-2,0) node[below]{$\varGamma_1$};
\draw[thick,blue]   (0,0) -- (2,0);
 \node at (1.3,1) {\tiny{{$u_0=-x^2+y^2$}}};
 \node at (-1.3,1) {\tiny{$u_0=- x^2+y^2$}};
  \node at (0,-1) { \tiny{$ u_0 =-x^2 -y^2 $}} ;
\end{tikzpicture}
 \begin{tikzpicture}[scale=1.1]
\draw[ultra thin,->] (-2.5,0) -- (2.5,0) node[anchor=west] {$x$} ;
\draw[ultra thin, ->] (0,-2.2) -- (0,2.2) node[anchor=south] {$y$} ;
\draw[thick,red]  (0,0) -- (0,2)node[right]{$\varGamma_2$} ;
\draw[thick,red]   (0,0) -- (0,-2);
\draw[thick,blue]  (0,0) -- (-2,0) node[below]{$\varGamma_1$};
\draw[thick,blue]   (0,0) -- (-2,0);
 \node at (1.3,0.3) {\tiny{{$u_0=x^2+y^2$}}};
 \node at (-1.3,1) {\tiny{$u_0=- x^2+y^2$}};
 \node at (-1.3,-1) { \tiny{$ u_0 =-x^2 +y^2 $}};
\end{tikzpicture}
\end{center}
\begin{center} 
\footnotesize{\textbf{ Figure 3.1:}  Examples of halfspace solutions. } 
\end{center}
\end{samepage}
Now let us look at the following two explicit solutions, which  obviously are not halfspace solutions. 

\begin{samepage}
\begin{center} \label{cross}
 \begin{tikzpicture}[scale=1.1]
\draw[ultra thin,->] (-2.5,0) -- (2.5,0) node[anchor=west] {$x$} ;
\draw[ultra thin, ->] (0,-2.2) -- (0,2.2) node[anchor=south] {$y$} ;
\draw[thick,red]  (0,0) -- (2,0)node[above]{$\varGamma_2$} ;
\draw[thick,red]   (0,0) -- (0,2);
\draw[thick,blue]  (0,0) -- (-2,0) node[below]{$\varGamma_1$};
\draw[thick,blue]   (0,0) -- (0,-2);
 \node at (1.3,1) {\tiny{{$u_0= x^2+y^2$}}};
 \node at (-1.3,-1) {\tiny{$u_0=- x^2-y^2$}};
 \node at (1.3,-1) {\tiny{ $ u_0 =x^2-y^2 $}} ;
  \node at (-1.3,1) { \tiny{$ u_0 =-x^2 +y^2 $}} ;
\end{tikzpicture}
 \begin{tikzpicture}[scale=1.1]
\draw[ultra thin,->] (-2.2,0) -- (2.2,0) node[anchor=west] {$x$} ;
\draw[ultra thin,->] (0,-2.2) -- (0,2.2) node[anchor=south] {$y$} ;
\draw[thick,red]  (0,0)  -- (2,2)node[above]{\tiny{$y=x$}}node[below]{$\varGamma_2$} ;
\draw[thick,red]   (0,0)  -- (2,-2) node[below]{\tiny{$y=-x$}} ;
\draw[thick,blue]  (0,0) -- (-2,2) node[above]{\tiny{$y=-x$}} ;
\draw[thick,blue]   (0,0)  -- (-2,-2) node[below]{\tiny{$y=x$}} node[above]{$\varGamma_1$} ;
 \node at (1.8,0.5) {\tiny{$u_0= x^2+y^2$}};
 \node at (-1.8,0.5) {\tiny{$u_0=- x^2-y^2$}};
 \node at (0,1) {\tiny{ $ u_0 =2xy $}} ;
  \node at (0,-1) {\tiny{ $ u_0 =-2xy $}} ;
\end{tikzpicture}
\end{center}
\begin{center}
 \footnotesize{\textbf{ Figure 3.2:}  New, interesting type of solutions.} 
\end{center}
\end{samepage}

We see that $ \varGamma = \varGamma_1 \cup \varGamma_2 $ 
consists of two lines meeting at right angles, and 
$ \varGamma_1 \cap \varGamma_2 =\{0\} = \varLambda$.
Actually there are many more solutions, for example consider the following global solutions

\begin{samepage}
\begin{center}
\begin{tikzpicture}[scale=1.1]
\draw[ultra thin,->] (-2.2,0) -- (2.2,0) node[anchor=west] {$x$} ;
\draw[ultra thin, ->] (0,-2.2) -- (0,2.2) node[anchor=south] {$y$} ;
\draw[thick,blue]  (0,0) node[left]{\tiny{$ \theta= \frac{\pi}{2}$}}-- (1.1,2.2) node[above]{\tiny{$y=2x$}};
\draw[thick,blue]   (0,0) -- (-1.1,2.2)node[above]{\tiny{$y=-2x$}}node[below]{$\varGamma_1$};
\draw[thick,red]  (0,0) -- (2.2,-1.1)node[below]{\tiny{$y=-\frac{x}{2}$}} ;
\draw[thick,red]   (0,0) -- (-2.2,-1.1)node[below]{\tiny{$y=\frac{x}{2}$}}node[above]{$\varGamma_2$};
 \node at (0,2) {\tiny{$u_0= -x^2-y^2$}};
 \node at (0,-1.5) {\tiny{$u_0=x^2+y^2$}};
 \node at (1.6,0.5) {\tiny{ $u_0= \frac{-3x^2+ 8 xy+3y^2}{5}$}} ;
  \node at (-1.6,0.5) { \tiny{$u_0= \frac{-3x^2-8 xy+3y^2}{5}$}} ;
\end{tikzpicture}
\begin{tikzpicture}[scale=1.1]
\draw[ultra thin,->] (-2.2,0) -- (2.2,0) node[anchor=west] {$x$} ;
\draw[ultra thin, ->] (0,-2.2) -- (0,2.2) node[anchor=south] {$y$} ;
\draw[thick,red]  (0,0) node[left]{\tiny{$ \theta= \frac{\pi}{2}$}}-- (1.1,2.2) node[above]{\tiny{$y=2x$}};
\draw[thick,red]   (0,0) -- (-1.1,2.2)node[above]{\tiny{$y=-2x$}} node[below]{$\varGamma_2$};
\draw[thick,blue]  (0,0) -- (2.2,-1.1)node[below]{\tiny{$y=-\frac{x}{2}$}} node[above]{$\varGamma_1$};
\draw[thick,blue]   (0,0) -- (-2.2,-1.1)node[below]{\tiny{$y=\frac{x}{2}$}};
 \node at (0,2) {\tiny{$u_0= x^2+y^2$}};
 \node at (0,-1.5) {\tiny{$u_0=- x^2-y^2$}};
 \node at (1.6,0.5) {\tiny{ $u_0= \frac{-3x^2+ 8 xy+3y^2}{5}$}} ;
  \node at (-1.6,0.5) { \tiny{$u_0= \frac{-3x^2-8 xy+3y^2}{5}$}} ;
\end{tikzpicture}
\end{center}
\begin{center}
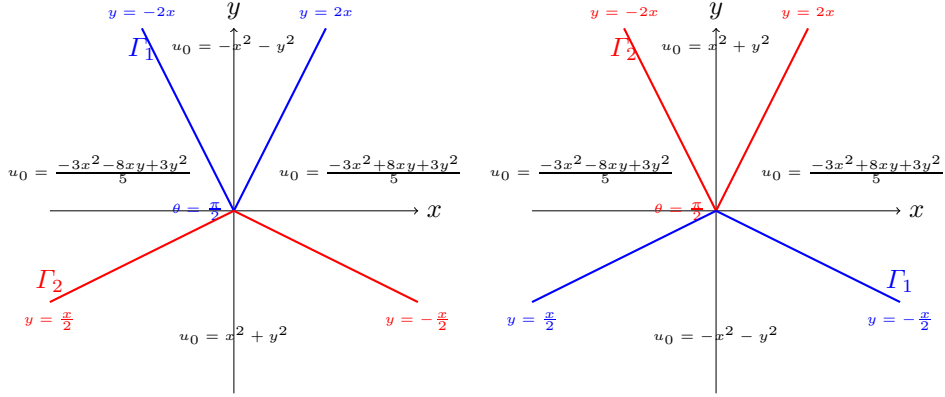
 \footnotesize{\textbf{ Figure 3.3:} In this example we see that the cone $\{ u_0=p^1\}$ ( $\{ u_0 =p^2\}$)
does not have a fixed opening angle. Actually the opening angle can take any value in the closed interval $ [0,\pi]$}. 
\end{center}
\end{samepage}
We see that in all the examples discussed above there is one common property:
in the halfplane $ x\geq0$ the lines $ \varGamma_1$ and $ \varGamma_2$
intersect at a right angle, later on we will provide a rigorous argument for this.

\par
Let us study two more examples, where the free boundary shows a different behaviour. 
\begin{example}  \label{ex2}
Let   $ p^1(x,y) =-x^2-y^2$ and
 $ p^2(x,y)=2x^2+2 y^2$. Assume that $ u^0$ is a homogeneous of degree two solution to the
 double obstacle problem with obstacles $ p^1$ and $ p^2$ in $\mathbb{R}^2$, then  $ \varGamma_2= \{0\}$. In this 
 case if  a  blow-up is not a polynomial, then it is a halfspace solution
 corresponding to $p^1 $.
\end{example}
It is easy to verify that there is no second order harmonic polynomial in $\mathbb{R}^2$, satisfying
$p^1 \leq q \leq p^2$ and such that the polynomials $ p^2- q$
and $ q-p^1 $ both have roots of multiplicity two.
Furthermore, in this case there are no halfspace solutions corresponding to $p^2$.

\begin{samepage}
\begin{center}
 \begin{tikzpicture}[scale=1.1]
\draw[ultra thin,->] (-2.5,0) -- (2.5,0) node[anchor=west] {$x$} ;
\draw[ultra thin, ->] (0,-2.2) -- (0,2.2) node[anchor=south] {$y$} ;
\draw[thick,blue]  (-2,-2)node[below]{\tiny{$y=x$}} -- (2,2) node[below]{$\varGamma_1$};
 \node at (-1.3,1) {\tiny{{$u_0= -x^2-y^2$}}};
 \node at (1.3,-1) {\tiny{$u_0=- 2xy$}} ;
\end{tikzpicture}
 \begin{tikzpicture}[scale=1.1]
\draw[ultra thin,->] (-2.5,0) -- (2.5,0) node[anchor=west] {$x$} ;
\draw[ultra thin, ->] (0,-2.2) -- (0,2.2) node[anchor=south] {$y$} ;
\draw[thick,blue]  (-2,2)node[above]{\tiny{$y=-x$}} -- (2,-2) node[below]{$\varGamma_1$};
 \node at (-1.3,-1) {\tiny{{$u_0= -x^2-y^2$}}};
 \node at (1.3,1) {\tiny{$u_0= 2xy$}} ;
\end{tikzpicture}
\end{center}
\begin{center}
 \footnotesize{\textbf{ Figure 3.4:} Examples of halfspace solutions. }
\end{center}
\end{samepage}

\begin{samepage}
\begin{example}  \label{ex3}
The following functions are  homogeneous global solutions to the double obstacle problem with 
 $ p^1(x,y) =-x^2-y^2$ and
 $ p^2(x,y)=2 x^2$.
\end{example}

\begin{center}
 \begin{tikzpicture}[scale=1.1]
\draw[ultra thin,->] (-2.5,0) -- (2.5,0) node[anchor=west] {$x$} ;
\draw[ultra thin, ->] (0,-2.5) -- (0,2.5) node[anchor=south] {$y$} ;
\draw[thick,blue]  (0,0) -- (1.4,2.4)node[above]{$\varGamma_1$};
\draw[thick,blue]  (0,0) -- (-1.4,2.4)node[above]{\tiny{$ y=-\sqrt{3}x$} };
\draw[thick,red]   (0,0) -- (1.4,-2.4)node[above]{$\varGamma_2$};
\draw[thick,red]   (0,0) -- (-1.4,-2.4)node[below]{\tiny{$ y=\sqrt{3}x$}};
 \node at (0,1.6) {\tiny{{$u_0=- x^2-y^2$}}};
 \node at (1.5,0.3) {\tiny{$u_0=\frac{x^2 - 2\sqrt{3}xy-y^2}{2}$}};
 \node at (-1.5,-0.3) {\tiny{$u_0= \frac{x^2+ 2\sqrt{3}xy-y^2}{2}$}} ;
  \node at (0,-1.4) {\tiny{ $ u_0 =2x^2 $} } ;
\end{tikzpicture}
 \begin{tikzpicture}[scale=1.1]
\draw[ultra thin,->] (-2.5,0) -- (2.5,0) node[anchor=west] {$x$} ;
\draw[ultra thin, ->] (0,-2.5) -- (0,2.5) node[anchor=south] {$y$} ;
\draw[thick,blue]  (0,0) -- (1.4,2.4)node[above]{$\varGamma_1$};
\draw[thick,blue]  (0,0) -- (1.4,-2.4)node[below]{\tiny{$ y=-\sqrt{3}x$} };
\draw[thick,red]   (0,0) -- (-1.4,-2.4)node[below]{$\varGamma_2$};
\draw[thick,red]   (0,0) -- (-1.4,2.4)node[above]{\tiny{$ y=\sqrt{3}x$}};
 \node at (1.2,0.3) {\tiny{{$u_0=- x^2-y^2$}}};
 \node at (0,2.1) {\tiny{$u_0=\frac{x^2 -2 \sqrt{3}xy-y^2}{2}$}};
 \node at (0,-2.1) {\tiny{$u_0= \frac{x^2+2 \sqrt{3}xy-y^2}{2}$}} ;
  \node at (-1.2,0.3) {\tiny{ $ u_0 =2x^2 $} } ;
\end{tikzpicture}
\end{center}
\begin{center}
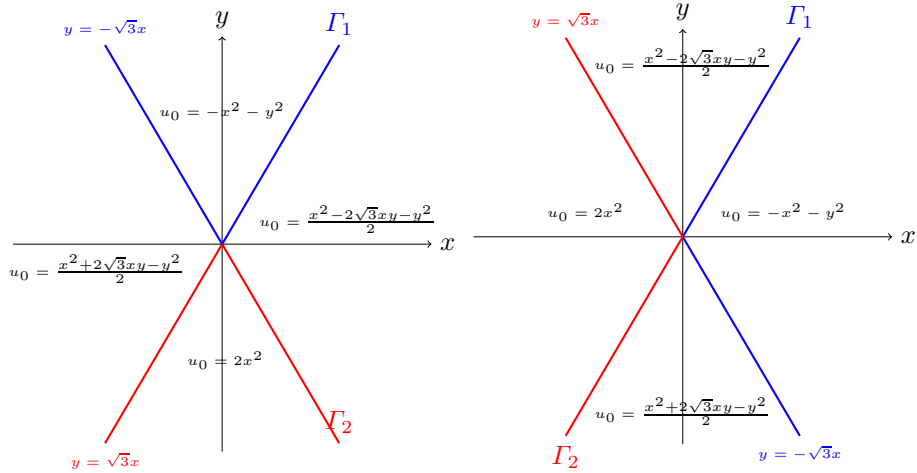
 \footnotesize{\textbf{ Figure 3.5:} The noncoincidence set is a cone with an opening angle 
 $ 2\pi/3$ or $ \pi/3$. }
\end{center}
\end{samepage}

\subsection{Double-cone solutions}

Let $ p^1 \leq p^2$ be given polynomials,
\begin{equation} \label{p1p2}
  p^i(x) \equiv a_i x_1^2 + 2b_i x_1 x_2 +c_i x_2 ^2, \textrm{ for } i=1,2,
\end{equation}

Consider the following normalised double obstacle problem in $\mathbb{R}^2$ with obstacles $p^1$, $p^2$;
\begin{equation} \label{equationsection3}
 p^1 \leq u \leq p^2, ~\Delta u = \lambda_1 \chi_{\{u= p^1\}}+\lambda_2 \chi_{\{u= p^2\}},
\end{equation}
where
\begin{equation}\label{lambda120}
 \lambda_1:=\Delta p^1 =  2(a_1+c_1) < 0 \textrm{ and } \lambda_2:=\Delta p^2=2(a_2+c_2) > 0.
\end{equation}

\par
We saw in Example \ref{ex1} and Example \ref{ex3} that for the double obstacle problem there exist global
solutions for which
the coincidence sets $ \{u=p^1\}$ and $ \{u=p^2\}$ are halfcones with a common vertex at the origin.

\begin{definition}\label{doublecone}
Let $ u$ be a global solution 
to the normalised double obstacle problem with  obstacles $ p^1 \leq p^ 2$.
We say that $u$ is a double-cone solution, if 
both $ \{ u=p^1 \}$ and $ \{ u=p^2\}$ are  halfcones with a common vertex. 
\end{definition}

\begin{remark}
Definition \ref{doublecone} is applicable also in higher dimensions. In the last section we will show the existence of three-dimensional double-cone solutions. 
\end{remark}

\par
In this section our aim is to describe the possible blow-ups for a solution to the double obstacle problem in
$ \mathbb{R}^2$.
In particular, we are interested to study the case when the double-cone solutions do exist.
It is easy to verify that if  $ \lambda_1=0$ or $\lambda_2=0$, there are no double-cone solutions,
explaining our assumption \eqref{lambda120}.

\par
A simple calculation shows that if $ p^1=p^2$ on a line, then there are no double-cone solutions.
Hence we assume that $ p^1$ and $ p^2$
meet only at the origin, in other words the matrix  $D^2 (p^2-p^1)$ is positive definite.

\par
Without loss of generality we may assume that $ b_1=b_2=b$ in \eqref{p1p2}. Otherwise, if $ b_2-b_1 \neq 0$,
we can rotate the coordinate system with an angle $ \theta$,
$\frac{\cos 2\theta}{\sin 2\theta }= \frac{a_2-a_1-c_2+c_1}{2(b_1-b_2)} $, and obtain $ b_1=b_2$ in the new system.
Furthermore, we may subtract a harmonic polynomial $ h(x)= 2bx_1 x_2$ from $p^1$,  $p^2$ and $ u$, then
consider instead the polynomials $p^1-h$ and $p^2-h$, thus obtaining $b=0$. Instead of $ u$, we are
studying the solution $ u-h$, but still call it $u$.

\par
We saw that it is enough to study the blow-up solutions of the double obstacle problem with obstacles 
having the form
\begin{equation} \label{p12}
 p^1(x) =a_1 x_1^2 +c_1 x_2^2 ~~\textrm{ and }~~
  p^2(x) = a_2 x_1^2 +c_2 x_2^2.
\end{equation}
According to our assumption,
the matrix $ A:=D^2( p^2- p^1 )$ is positive definite, hence
\begin{equation} \label{a2a1}
 a_2 >a_1,~~~ c_2 >c_1.
\end{equation}
and by \eqref{lambda120},
\begin{equation}\label{lambda12}
 a_1+c_1 < 0, \textrm{ and } a_2+c_2>0.
\end{equation}

\par
If $u$ is a double-cone solution in  $\mathbb{R}^2 $, then the
noncoincidence set $ \Omega_{12}=\{p^1< u <p^2 \}$ consists of two halfcones $ \mathcal{S}_1$ and $ \mathcal{S}_2$, 
having a common vertex. So the expression ''double-cone'' may refer to the cones $\mathcal{S}_i$ as well.
The following lemma is the  main step to the investigation of double-cone solutions
in $ \mathbb{R}^2$.

\begin{lemma} \label{lemmaqs}
 Let  $ p^1( x) = a_1 x_1^2+ c_1 x_2^2$ and $ p^2( x) = a_2 x_1^2 + c_2 x_2^2$ be given polynomials, 
satisfying \eqref{a2a1} and \eqref{lambda12}.
Assume that there exists a pair $ (q, \mathcal{S})$, where $\mathcal{S}$ is
an open sector in $\mathbb{R}^2$, with the edges lying on the lines $ x_2=mx_1$ and $x_2=kx_1$, and $ q$ is a 
harmonic homogeneous of degree two function in $\mathcal{S}$.  Moreover, assume that 
\begin{equation} \label{qS}
 p^1 \leq q \leq p^2  \textrm{ in } \mathcal{S}, 
 \end{equation}
 and  the following boundary conditions hold;
\begin{equation} \label{bdqp1}
  q- p^1 = 0 ,  \nabla (q- p^1 )=0 \textrm{ on } x_2=mx_1,
\end{equation}
and
\begin{equation} \label{bdqp2}
 q- p^2 = 0 ,  \nabla (q- p^2 )=0 \textrm{ on } x_2=kx_1.
\end{equation}
Then $ q =\alpha x_1^2+ 2  \beta x_1 x_2 - \alpha x_2^2$, where $ \alpha$ and $ \beta$ are real numbers
solving
\begin{equation}\label{beta2}
  \beta^2=-(\alpha-a_1) ( \alpha +c_1)= - (\alpha-a_2) ( \alpha +c_2) \geq 0,
\end{equation}
\begin{equation} \label{alpha}
 \max( a_1,-c_2) \leq \alpha \leq \min( a_2,-c_1) ,
\end{equation}
and
\begin{equation} \label{a1c1a2c2}
 \alpha( c_1-a_1 -c_2 + a_2)= a_1 c_1- a_2 c_2.
\end{equation}
The numbers $ m $ and $ k $ are given by 
\begin{equation} \label{km}
 m= \frac{\beta}{\alpha + c_1}~ \textrm{ and }~ k= \frac{\beta}{c_2+\alpha}.
\end{equation}
Furthermore, the coefficients of $ p^1$ and $ p^2$ satisfy the following inequality 
\begin{equation} \label{necessary}
  (a_1+c_2)(c_1+ a_2) \leq 0.
\end{equation}

\end{lemma}

\begin{proof}
 Let us note that harmonic homogeneous of 
degree two functions in a sector are second degree polynomials of the form
$q(x)= \alpha x_1^2+ 2  \beta x_1 x_2 - \alpha x_2^2$, 
where $ \alpha$ and $ \beta$ are real numbers. By assumption \eqref{qS}, 
\begin{equation*}
\begin{aligned}
q-p^1=(\alpha-a_1) x_1^2+ 2 \beta x_1x_2 -(\alpha+c_1)x_2^2 \geq 0, \textrm{ and } \\
p^2-q= (a_2-\alpha)x_1^2-2\beta x_1x_2+(c_2+\alpha)x_2^2 \geq 0 ~\textrm{ in } ~\mathcal{S} .
\end{aligned}
\end{equation*}
Denote by $t = \frac{x_2}{x_1}$, and observe that \eqref{bdqp1}
implies that
the following quadratic polynomial 
\begin{align*}
\frac{q- p^1}{x_1^2}= -(\alpha+c_1)t^2 + 2\beta t+\alpha-a_1
\end{align*}
has a multiple root at the point $ t=m$. 
By an elementary calculation we obtain  
\begin{equation*}
 \beta^2=-(\alpha-a_1)(\alpha+c_1), \textrm{ and }
\frac{q- p^1}{x_1^2}=- (\alpha + c_1) (t-m)^2.
\end{equation*}
Hence the inequality $q- p^1 \geq 0$ in $ \mathcal{S}$ implies $ q-p^1 \geq 0 $ in 
$ \mathbb{R}^2$. Therefore we may conclude that
\begin{equation} \label{p1q}
 -\alpha-c_1 \geq 0,~\alpha-a_1 \geq 0,~ \beta^2= -(\alpha-a_1) (\alpha + c_1), \textrm{ and }
 m= \frac{\beta}{\alpha+c_1}.
\end{equation}

\par
Similarly,  \eqref{bdqp2} implies that
the following quadratic polynomial 
\begin{align*}
\frac{p^2-q}{x_1^2}= (c_2+\alpha)t^2 + -2\beta t+a_2-\alpha
\end{align*}
has a multiple root at the point $ t=k$. 
Hence $ \beta^2= (a_2-\alpha)(c_2+\alpha)$, and the inequality $p^2-q \geq 0$ in $ \mathcal{S}$ implies $ p^2-q \geq 0 $ in 
$ \mathbb{R}^2$. Therefore, by a similar argument as the one leading to \eqref{p1q}, we get
\begin{equation} \label{p2q}
c_2+\alpha\geq 0,~a_2-\alpha \geq 0, ~\beta^2= (a_2-\alpha)(c_2+\alpha),  \textrm{ and } 
k = \frac{\beta}{c_2+\alpha}.
\end{equation}
Let us also observe that if $ \alpha=-c_1$, then $p^1(x)=q(x)$ implies $ x_1=0$, 
similarly, if $ \alpha=-c_2$, then $p^2(x)=q(x)$ implies $ x_1=0$. Hence \eqref{km} makes sense even if 
$ \alpha =-c_1$ or $ \alpha=-c_2$.

\par
 Assuming that there exists $ (q, \mathcal{S})$ satisfying \eqref{qS},\eqref{bdqp1},\eqref{bdqp2}, 
 we derived \eqref{p1q} and \eqref{p2q}, which in particular imply  \eqref{beta2}, \eqref{alpha} and \eqref{km}.
 It follows from  \eqref{beta2}, that 
 \begin{equation*}
  \alpha^2 -a_1 \alpha +c_1 \alpha -a_1c_1= \alpha^2 -a_2 \alpha +c_2 \alpha -a_2 c_2,
 \end{equation*}
hence $ \alpha $ solves equation \eqref{a1c1a2c2}.
 As we see equation \eqref{a1c1a2c2} is contained in \eqref{beta2}, we stated \eqref{a1c1a2c2} only for the future
 references.

 \par
It remains to prove the  inequality 
\eqref{necessary}, which is a necessary condition for the existence of  $ \alpha, \beta $, thus 
for $ (q,\mathcal{S})$.
We discuss two cases.
$i)$ If 
\begin{equation} \label{zero}
c_1-a_1 -c_2 + a_2  =0,
\end{equation}
it follows from equation \eqref{a1c1a2c2} that 
\begin{equation}\label{one}
a_2c_2=a_1 c_1.
\end{equation} 
 If $ a_1=0$, then $ a_2 <0$ by \eqref{a2a1}, therefore $c_2=0$, and \eqref{necessary} holds.
Otherwise, if $ a_1 \neq 0$, let 
$ a_2= l a_1 $, then $ l\neq 1$ by \eqref{a2a1}. Hence $ c_1 = l c_2$ according to \eqref{one}. Now 
\eqref{zero} implies that $(l-1)(a_1+c_2)=0$, since $ l\neq 1$, we obtain $ a_2+ c_1=0$, and \eqref{necessary}
holds.

\par
$ii)$ If $ c_1-a_1 -c_2 + a_2 \neq 0 $, then 
equation \eqref{a1c1a2c2} implies that
 \begin{equation} \label{Alpha}
  \alpha = \frac{a_2 c_2-a_1 c_1}{ c_2+a_1-a_2-c_1}.
 \end{equation}
 By a direct computation we see that 
 \begin{align*}
  \alpha-a_1=\frac{(a_2-a_1)(a_1+c_2)}{c_2+a_1-a_2-c_1}, ~ \textrm{ and }
  \alpha+c_1=\frac{(c_2-c_1)(a_2+c_1)}{c_2+a_1-a_2-c_1}, 
 \end{align*}
by \eqref{beta2} 
\begin{equation*}
 \beta^2= -\frac{(a_2-a_1)(c_2-c_1)(a_1+c_2)(a_2+c_1)}{(c_2+a_1-a_2-c_1)^2} \geq 0.
\end{equation*}
Taking into account  \eqref{a2a1}, we obtain the desired inequality, \eqref{necessary}.
 
\end{proof}

\par
Let us observe that if $ u_0$ is a double-cone solution (Definition \ref{doublecone}), then there exist
$ (q_1, \mathcal{S}_1)$ and  
$(q_2, \mathcal{S}_2)$ 
as in Lemma \ref{lemmaqs}, such that $ \mathcal{S}_1\cap \mathcal{S}_2=\emptyset $, and
\begin{equation} \label{s1s2}
 u_0= q_1~ \textrm{ in }~ \mathcal{S}_1 \textrm{ and } ~ u_0= q_2 \textrm{ in }~ \mathcal{S}_2.
\end{equation}

\par
According to Lemma \ref{lemmaqs}, the inequality \eqref{necessary} is a necessary condition for the existence 
of double-cone solutions, in the next theorem we will discuss if \eqref{necessary} is also a sufficient condition.

\begin{theorem} \label{theorem1}{\textbf{(The existence of double-cone solutions)}} \\
 Let $ u_0$ be a homogeneous of degree two global solution 
to the double obstacle problem with obstacles 
\begin{equation*}
  p^1( x) = a_1 x_1^2+ c_1 x_2^2~\textrm{ and } ~ p^2( x) = a_2 x_1^2 + c_2 x_2^2,
\end{equation*}
satisfying \eqref{lambda12} and \eqref{a2a1}.
If $ u_0$ is neither a polynomial nor a 
halfspace solution, then it is a double-cone solution.
\begin{enumerate}
\item[Case 1)] If $ a_2+c_1= c_2 +a_1=0$, then
there are infinitely many double-cone solutions.  Each of the  
 cones $\mathcal{S}_1$ and $\mathcal{S}_2$ in \eqref{s1s2} has an opening angle $ \vartheta= \pi /2 $.

\item[Case 2)] If $ (a_1+c_2)(c_1+ a_2) < 0$, then there 
exist  four double-cone solutions.  Furthermore, the
opening angle of $\mathcal{S}_i$, denoted by  $\vartheta_i$, satisfies
\begin{equation}
\vartheta_1=\vartheta_2=\vartheta, \textrm{ and }  \cos^2 \vartheta = {\frac{(a_1+c_2)(a_2+c_1)}{(a_1+c_1)(a_2+c_2)}} \in (0,1).
\end{equation}

\item[Case 3)] If $ (a_1+c_2)(c_1+ a_2) \geq 0$, and $a_1+c_2 \neq 0$ or $a_2+c_1\neq 0$, then there are no double-cone solutions. 

\end{enumerate}
\end{theorem}

\begin{proof}

\par
If $ u_0$ is neither a polynomial nor a 
halfspace solution, then there exists a pair $ (q, \mathcal{S})$, such that 
\begin{equation}
 u_0= q \textrm{ in } \mathcal{S},~u_0= p^1 \textrm{ on } \{ x_2=mx_1\}, ~u_0=p^2 \textrm{ on } \{ x_2=k x_1\}
\end{equation}
where $S$ is
a sector in $\mathbb{R}^2$, with edges lying on the lines $\{ x_2=mx_1\} $ and $ \{x_2=kx_1\}$, and  $ q$ is a 
harmonic homogeneous degree two function in $\mathcal{S}$, satisfying \eqref{qS}.
Moreover, since $ u_0 \in C^{1,1}$, we obtain $ \nabla q=\nabla p^1$ on $\{ x_2=mx_1\} $ and 
$\nabla  q= \nabla p^2$ on
$ \{x_2=kx_1\}$. Hence $q$ takes
boundary conditions \eqref{bdqp1} and \eqref{bdqp2}  on
$ \partial \mathcal{S} \subset \{ x_2=mx_1\} \cup \{x_2=kx_1\}$, and 
therefore $ (q, \mathcal{S}) $ satisfies the assumptions in Lemma \ref{lemmaqs}.

\par
According to Lemma \ref{lemmaqs}, $ q =\alpha x_1^2+ 2  \beta x_1x_2 - \alpha x_2^2$, where $ \alpha$ and $ 
\beta$ are real numbers
solving \eqref{beta2} and \eqref{alpha}. The  numbers $ m $ and $ k $, describing the sector $\mathcal{S}$, are 
given by 
\eqref{km}.

\par
 We are looking for all possible pairs $ (q, \mathcal{S})$ in terms of the parameter $ \alpha$. 
Given $\alpha $, satisfying \eqref{alpha} and \eqref{a1c1a2c2},  we can find $\pm \beta$ from 
 equation \eqref{beta2}.
By equation \eqref{km} we can identify the corresponding sectors $\mathcal{S}$.

\par
Let us split the discussion into several cases  in order to study 
the existence of solutions to the equation \eqref{a1c1a2c2} in variable $\alpha$, 
satisfying inequality \eqref{alpha}.

\textit{ \textbf{ Case 1)}} If $ a_2+c_1=c_2 +a_1=0$, 
as in  Example \ref{ex1}. Then obviously equation \eqref{a1c1a2c2} becomes an identity. Hence in this case 
$ \alpha $ can be any number satisfying \eqref{alpha}, that is $  a_1 \leq \alpha \leq a_2$.
If $ \alpha=a_1$, then $ \beta =0$ in view of \eqref{beta2}, and according to \eqref{km},
$ \varGamma_1= \{ x_2=0\}$. In this case $\varGamma_2=\{x_1=0,x_2 \geq 0 \}$ or 
$\varGamma_2=\{x_1=0,x_2 \leq 0 \}$ . 
Analogously if  $ \alpha= a_2$ then $ \varGamma_2= \{ x_2=0\}$ and 
$\varGamma_1=\{x_1=0,x_2 \geq 0 \}$ or $\varGamma_2=\{x_1=0,x_2 \leq 0 \}$.
Hence we obtain halfspace solutions, see
Figure 3.1 with $ \alpha =-1=a_1$ and $ \alpha=1= a_2$. These can still be viewed as double-cone solutions, if we allow  the cone $ \mathcal{S}_2$ to be a halfline.

\par
 Now let us fix any $ a_1 < \alpha < a_2$. 
It follows from \eqref{p1q} and \eqref{p2q} that 
\begin{equation*}
\beta_{\pm} = \pm \sqrt{(\alpha- a_1)(a_2 -\alpha )} ,
\end{equation*}
and 
\begin{equation} \label{mk}
 m_{\pm}= \mp \sqrt{ \frac{\alpha - a_1}{a_2-\alpha}}, ~~ k_{\pm}= \pm \sqrt{ \frac{a_2- \alpha}{\alpha -a_1}}.
\end{equation}
Let us note that $ m^\pm k^\pm =-1$, and therefore the lines $ x_2=m^\pm x_1$ and $ x_2= k^\pm x_1$ are perpendicular.
Thus for a fixed $ a_1<\alpha < a_2$ we obtain two polynomials 
\begin{equation} \label{q12}
 q_+ := \alpha x_1^2+ 2 \beta_+ x_1x_2 - \alpha x_2^2 ~\textrm{ and } ~
 q_- := \alpha x_1^2+ 2 \beta_- x_1x_2 - \alpha x_2^2.
\end{equation} 
Where $ q_+ = p^1$ if $ x_2=m_+ x_1 $, $ q_+=p^2 $ if $ x_2=k_+x_1$, and 
$ q_- = p^1$ if $ x_2=m_-x_1 $, $ q_-=p^2 $ if $ x_2=k_-x_1$. Hence for a fixed $\alpha $ there are two 
pairs $ (q_+, \mathcal{S}_1)$ and
$(q_-, \mathcal{S}_2) $ forming a single double-cone solution $ u_0$.
There are four different choices of disjoint sectors  $ \mathcal{S}_1$ and $ \mathcal{S}_2$, satisfying \eqref{mk}.
Therefore we obtain four different double-cone solutions for a fixed $ a_1<\alpha <a_2$.
Figure 3.3 illustrates two of them for $ \alpha = -\frac{3}{5}$. 

\par
In fact we obtain more double-cone solutions by "merging" two double-cone solutions corresponding to two different 
values of $ \alpha$. Consider the following example
 $ u_0 = x_1^2 \sgn (x_1) + x_2^2 \sgn (x_2)$ (see the left picture in Figure 3.2). Then $ q_1= -x_1^2+ x_2^2$
with $ \alpha_1=-1 $ and 
$ q^2= x_1^2-x_2^2$  with $ \alpha_2= 1$. 

\par
Fix any 
$ a_1 \leq \alpha_1 \neq \alpha_2 \leq a_2$, then there are 
four double-cone solutions corresponding to each of $ \alpha_i$.
From these double-cone solutions we obtain eight more double-cone solutions, such that 
$ \mathcal{S}_1\cap \mathcal{S}_2=\emptyset$, where
 $ q^i $, $ i=1,2$ 
can be either $q_+$ or  $ q_-$ corresponding to $ \alpha_i$. 
The solution $ u_0$ can be described graphycally as follows:
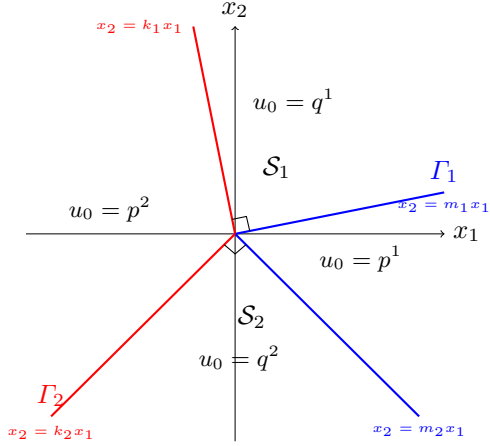
\begin{center}
 \begin{tikzpicture}[scale=1.1]
\draw[ultra thin,->] (-2.5,0) -- (2.5,0) node[anchor=west] {$x_1$} ;
\draw[ultra thin, ->] (0,-2.5) -- (0,2.5) node[anchor=south] {$x_2$} ;
\draw[thick,blue]  (0,0) -- (2.5,0.5)node[above]{$\varGamma_1$}node[below]{\tiny{$ x_2=m_1 x_1$}};
\draw[thick,red]  (0,0) -- (-0.5,2.5)node[left]{\tiny{$x_2=k_1 x_1$}} ;
\draw[thick,blue]   (0,0) -- (2.2,-2.2)node[below]{\tiny{$ x_2=m_2 x_1$}};
\draw[thick,red]   (0,0) -- (-2.2,-2.2)node[above]{$\varGamma_2$}node[below]{\tiny{$ x_2=k_2 x_1$}};
 \node at (0.7,1.6) {\small{{$u_0=q^1$}}};
 \node at (1.5,-0.3) {\small{$u_0=p^1$}};
 \node at (-1.5,0.3) {\small{$u_0= p^2 $} };
  \node at (0,-1.5) {\small{ $ u_0 =q^2 $} } ;
  \node at (0.5,0.8) {{$ \mathcal{S}_1$}};
   \node at (0.2,-1) {{$ \mathcal{S}_2$}};
      \draw  (-0.035,0.175) -- (0.135,0.21) -- (0.175,0.035);
         \draw  (-0.13,-0.13) -- (0,-0.242) -- (0.13,-0.13);
\end{tikzpicture}
\end{center}
This is a general example of a double-cone solution \eqref{s1s2}, where the polynomial $ q^1 $, and 
the numbers $m_1,k_1$ correspond to $ \alpha_1$, similarly
 $ q^2 $ and $m_2,k_2$ correspond to $ \alpha_2$. 
We conclude that, if $ u_0 $ is a double-cone solution then the cone $ \{ u_0 = p^1\}$
may have any opening angle $ \theta$, $0 <  \theta <\pi$, and
the cone $ \{ u_0 = p^2\}$
has an angle $ \pi-\theta$.
If $ \theta =0$ or $ \theta=\pi$, then $ u_0$
is a halfspace solution corresponding to $p^2$ or $ p^1 $ respectively.

\par
Finally, note that there are no homogeneous of degree two solutions $ u_0$ corresponding to three or more 
different values of $ \alpha$,
since $ u_0$  can have only an even number of $(q, \mathcal{S})$, and 
$ \mathcal{S}$ always has an opening angle $\pi/2$.

\textit{\textbf{ Case 2)}} 
If $ (a_1+c_2)(a_2+c_1) < 0$, then $ c_1-a_1 -c_2 + a_2 \neq 0 $, and the equation
\eqref{a1c1a2c2} has a unique solution, 
\begin{equation} \label{singlealpha}
 \alpha= \frac{a_2c_2-a_1c_1}{c_2+a_1-a_2-c_1}.
\end{equation}
From the inequality $ (a_1+c_2)(a_2+c_1) < 0$ it easily follows that 
\begin{equation}
\max( a_1,-c_2) < \alpha < \min( a_2,-c_1) .
\end{equation}
Referring to \eqref{beta2}, we can calculate  
\begin{equation*}
 \beta_{\pm}= \pm \sqrt{ (\alpha+ c_1)(a_1-\alpha)}=
 \pm \frac{\sqrt{ -(a_2-a_1)(c_2-c_1)(a_2+c_1)(a_1+c_2)}}{c_2+a_1-a_2-c_1}.
\end{equation*}
According to \eqref{km},
\begin{equation} \label{lines}
 m_{\pm}=\frac{\beta_{1,2}}{\alpha +c_1}=
 \mp\sqrt{-\frac{(c_2+a_1)(a_2-a_1)}{(a_2+c_1)(c_2-c_1)}}
 , ~~~k_{\pm}=\frac{\beta_{1,2}}{\alpha+c_2}= \pm \sqrt{-\frac{(c_1+a_2)(a_2-a_1)}{(a_1+c_2)(c_2-c_1)}}.
\end{equation}
Hence we obtain two harmonic polynomials $q_+$ and $ q_-$  and 
four combinations of disjoint $ \mathcal{S}_1$ and $ \mathcal{S}_2$. Since in this case 
$\alpha$ is a fixed number, given by \eqref{singlealpha},
there are only four double-cone solutions.

\par
Denote by $ \vartheta_i$ the opening angle of 
the cone $\mathcal{S}_i$, then it follows from \eqref{lines} that
\begin{align*}
 \cos \vartheta_i =\pm \frac{1+k_+ m_+}{\sqrt{1+(k_+)^2}\sqrt{1+(m_+)^2}} =
 \pm \frac{1+k_- m_-}{\sqrt{1+(k_-)^2}\sqrt{1+(m_-)^2}} = \\
=\pm \frac{ \frac{c_2-c_1 -a_2+ a_1}{c_2-c_1}}{
 \sqrt{\frac{(a_1+c_1)(c_2-a_2+a_1-c_1)}{(a_2+c_1)(c_2-c_1)}}
 \sqrt{\frac{(a_2+c_2)(c_2-a_2+a_1-c_1)}{(a_1+c_2)(c_2-c_1)}}}
 =\pm \sqrt{\frac{(a_1+c_2)(a_2+c_1)}{(a_1+c_1)(a_2+c_2)}}, \textrm{ for } i=1,2,
\end{align*}
hence $ \vartheta_1=\vartheta_2=\vartheta$, and 
\begin{equation*}
0<\cos^2\vartheta <1.
\end{equation*}
In Example \ref{ex3}, $a_1=c_1=-1$, $a_2=2$, $ c_2=0$, 
by a direct calculation we see that $ \alpha= \frac{1}{2}$, $\beta =\pm \frac{\sqrt{3}}{2}$, and
$ \vartheta_1=\vartheta_2 = \frac{\pi}{3}$ or $ \vartheta_1=\vartheta_2 = \frac{2 \pi}{3}$.

\par
\textit{ \textbf{ Case 3)}} If $ (a_1+c_2)(c_1+a_2)>0$, then the statement follows from the inequality \eqref{necessary} in  Lemma \ref{lemmaqs}.  
Otherwise if $ (a_1+c_2)(c_1+ a_2) = 0$, and $ a_1+c_2 \neq a_2+c_1 $, there are only halfspace solutions.
Indeed, assume that $ a_1+c_2=0$, and $a_2+c_1\neq 0$, then $\alpha =a_1=-c_2$ by \eqref{singlealpha}, and $ \beta =0$. 
Hence we obtain that $ u_0$ is a halfspace solution corresponding to $ p^1$.

\end{proof}

We say that a given polynomial $p$ has a sign, if $ p\not \equiv 0$, and $ p \geq 0$ (or $ p\leq 0$).
Let us rephrase Theorem \ref{theorem1} in a more compact form.

\begin{corollary} \label{core}
 Let $ p^i = a_i x_1^2+ c_i x_2^2 $ be given polynomials, satisfying 
 \eqref{lambda12} and \eqref{a2a1}.  There exist double-cone solutions 
 for the double obstacle problem with $ p^1, p^2$, if and only if the following polynomial
 \begin{equation} \label{p1+p2}
 P(x)=P(x_1,x_2) \equiv p^1(x_1, x_2) +p^2(x_2, x_1) = (a_1+c_2)x_1^2+(c_1+a_2)x_2^2 
 \end{equation}
has no sign. If $ P\equiv 0$, there are infinitely many double-cone
solutions. If $ P$ changes the sign, then there are four double-cone solutions, and if $P$ has a sign, there are none.
\end{corollary}

\par 
In other words, there exist double-cone solutions if and only  if the matrix $ D^2 P$ is neither positive nor 
negative definite.

\subsection{Halfspace solutions}

\par 
Now we proceed to the discussion on the existence of halfspace solutions in $ \mathbb{R}^2$, see \mbox{Definition \ref{halfspace}}. Let 
$ u \in C^{1,1}$  be such that $w= u- p^1$ is a halfplane solution to the obstacle problem  {$ \Delta w= -\lambda_1 \chi_{\{w>0\}}$},  we need to 
 check if $ u \leq p^2$ in $\mathbb{R}^2$.

\begin{lemma} \label{lemmahalfspace}
 Let $ p^i = a_i x_1^2+ c_i x_2^2 $   be given polynomials, satisfying 
 \eqref{lambda12} and \eqref{a2a1}.
Let $ u^1 \in C^{1,1}$  be a halfplane solution to the obstacle problem with obstacle $ p^1$,
\begin{equation}
u^1(x)=  \begin{cases}
  q(x)= \alpha x_1^2 +2\beta x_1x_2 -\alpha x_2^2 >p^1(x),  & \mbox{if } x_2 >k x_1 \\
  p^1(x),  & \mbox{if } x_2 \leq k x_1,
 \end{cases}
\end{equation}
where $\alpha, \beta$ and $k$ are real numbers. 
The function $u^1$ is a halfspace solution corresponding to $p^1$ (for the double obstacle problem with $p^1,p^2$), if and only if 
\begin{equation} \label{delta1}
\min(-c_1,a_2)\geq \alpha \geq \max(a_1,-c_2) , ~\delta_1(\alpha):= a_1c_1-a_2c_2+\alpha(a_1-c_1-a_2+c_2)\leq 0,
\end{equation}
and 
\begin{equation} \label{betak}
\beta^2=-(\alpha-a_1)(\alpha +c_1), ~k= \frac{\beta}{\alpha+c_1}.
\end{equation}
\par
Similarly, let
$ -u^2 \in C^{1,1}$  be a halfplane solution to the obstacle problem with obstacle $ -p^2$, 
\begin{equation}
u^2(x)=  \begin{cases}
  q(x)= \alpha x_1^2 +2\beta x_1x_2 -\alpha x_2^2 <p^2(x),  & \mbox{if } x_2 >m x_1 \\
  p^2(x),  & \mbox{if } x_2 \leq m x_1,
 \end{cases}
\end{equation}
Then  $u^2$ is a halfplane solution corresponding to $p^2$, if and only if 
\begin{equation} \label{delta2}
\min(-c_1,a_2)\geq \alpha \geq \max(a_1,-c_2) , ~\delta_2(\alpha):= -a_1c_1+a_2c_2-\alpha(a_1-c_1-a_2+c_2)\leq 0,
\end{equation}
and 
\begin{equation}
\beta^2=-(\alpha-a_2)(\alpha +c_2), ~m= \frac{\beta}{\alpha+c_2}.
\end{equation}
\end{lemma}

\begin{proof}
In order to show that $ u^1$ is a  halfspace solution (for the double obstacle problem) corresponding to $ p^1$, we need 
to verify that $ u^1 \leq p^2$ in $\mathbb{R}^2$.

\par
Let $ t = \frac{x_2}{x_1}$, and consider the following polynomial
 \begin{equation*}
f_1(t):=\frac{q(x)-p^1(x)}{x_1^2}= (-\alpha-c_1)t^2+2 \beta t+(\alpha-a_1) \geq 0.
\end{equation*}
Since $ u^1 \in C^{1,1}$, the polynomial $ f_1$  has a double root at $ t=k$, and therefore \eqref{betak} holds, with  $-c_1 \geq \alpha \geq a_1 $.
 Now let us prove \eqref{delta1}. 
Consider the polynomial
\begin{equation*}
f_2(t):=\frac{p^2(x)-q(x)}{x_1^2}= (c_2+\alpha)t^2-2 \beta t+(a_2-\alpha).
\end{equation*}
Since $ q \leq p^2$ on a halfplane, the polynomial $ f_2$ has to be  nonnegative. The latter is equivalent to the following inequalities
\begin{align*}
a_2 \geq \alpha \geq -c_2~\textrm{ and } \delta_1:=\beta^2-(a_2-\alpha)(c_2+\alpha)\leq 0.
\end{align*}
Taking into account \eqref{betak}, we obtain the desired inequality \eqref{delta1}.

\par
The corresponding statement for $ u^2$ can be proved similarly.

\end{proof}

\begin{theorem} \label{theoremhalfspace} {\textbf{(The existence of halfspace solutions)}} \\
Let the assumptions of Theorem \ref{theorem1} hold.
\begin{enumerate}
\item[Case 1)] If $ a_2+c_1= c_2 +a_1=0$, then
there are infinitely many rotational invariant halfplane solutions corresponding to $p^1$ ( $ p^2$).

\item[Case 2)] If $ (a_1+c_2)(c_1+ a_2) < 0$, then there 
exist infinitely many halfplane solutions corresponding to $p^1$ ( $p^2$), and $\varGamma_i$ always remains inside a fixed cone, thus 
halfplane solutions are not rotationally invariant on the entire plane.

\item[Case 3)] $ (a_1+c_2)(c_1+ a_2) \geq 0$, and $a_1+c_2 \neq 0$ or $a_2+c_1\neq 0$. If $ a_1+c_2\geq 0, a_2+c_1>0$ then there are 
 infinitely many rotationally invariant halfspace solutions corresponding to $p^1$, and at most two halfspace solutions corresponding to $p^2$.
 If $ a_1+c_2\leq 0, a_2+c_1<0$ then there are 
 infinitely many rotational invariant halfspace solutions corresponding to $p^2$, and at most two halfspace solutions corresponding to $p^1$.
 \end{enumerate}
\end{theorem}

\begin{proof}
\textit{ \textbf{ Case  1)}}  
 In this case $ \delta_1(\alpha)=\delta_2(\alpha)=0$ for any $\alpha$.
 By Lemma \ref{lemmahalfspace}, $-c_1\geq \alpha \geq a_1$ can be any number, and  $k $ ($m$)
take any value in the closed interval $\left[-\infty, \infty \right]$.
Hence $\varGamma_1$ ($\varGamma_2$) can be any line on the plane.  Furthermore, if  $u^1$
is a halfspace solution corresponding to $p^1$, then $\varGamma_2=\{u^1=p^2\}$ is a halfline, perpendicular to $\varGamma_1$.

\textit{ \textbf{ Case  2)}}  
Without loss of generality we may assume that $a_1+c_2>0$ and $ a_2+c_1<0$. 
Employing Lemma \ref{lemmahalfspace}, \eqref{delta1} together with \eqref{a2a1} and \eqref{lambda12}, we obtain
the following chain of inequalities
\begin{equation} \label{alphahalfspace}
-c_2<a_1 \leq \alpha \leq \frac{a_2c_2-a_1c_1}{ a_1-c_1-a_2+c_2} < a_2<-c_1,
\end{equation}
and therefore 
\begin{equation} \label{kGamma1}
 | k | = \sqrt{ \frac{\alpha-a_1}{-\alpha-c_1}} \leq \sqrt{-\frac{(a_2-a_1)(a_1+c_2)}{(c_2-c_1)(a_2+c_1)}}:= K>0.
\end{equation}
It follows from Lemma \ref{lemmahalfspace} that for any $\alpha$ satisfying \eqref{alphahalfspace} there are four halfspace solutions $u^1$ corresponding to $p^1$, and 
\begin{equation*}
 \varGamma_1=\partial{\{u^1>p^1\}}=\{x: x_2=\pm k x_1\} \subset \{x: | x_2| \leq K |x_1|\}
 \end{equation*}
  by \eqref{kGamma1}, and $ \varGamma_2=\{0\}$ if $a_1< \alpha< \frac{a_2c_2-a_1c_1}{a_1-c_1-a_2+c_2}$.

\par
Let us also discuss the halfspace solutions corresponding to $ p^2$.
According to Lemma \ref{lemmahalfspace}, \eqref{delta2},
\begin{equation*} 
-c_2<a_1 < \frac{a_2c_2-a_1c_1}{ a_1-c_1-a_2+c_2} \leq \alpha \leq a_2<-c_1,
\end{equation*}
and 
\begin{equation*}
 | m | = \sqrt{ \frac{a_2-\alpha}{\alpha+c_2}} \leq \sqrt{-\frac{(a_2-a_1)(a_2+c_1)}{(c_2-c_1)(a_1+c_2)}}:= M >0.
\end{equation*}
Therefore we obtain infinitely many halfspace solutions $u^ 2$ corresponding to $p^2$, and 
\begin{equation*}
 \varGamma_2=\partial{\{u^2>p^2\}} \subset \{x: | x_2| \leq M |x_1|\}.
 \end{equation*}

\textit{ \textbf{ Case 3)}}  
Assume that  $a_1+c_2\geq 0$  and $ c_1+a_2 \geq 0$, then
\begin{align*}
\delta_1= a_1c_1-a_2c_2 +\alpha(a_1+c_2)-\alpha (a_2+c_1) \leq \\
 a_1c_1 -a_2c_2-c_1(a_1+c_2) -a_1(a_2+c_1)=
-(a_1+c_2)(a_2+c_1)\leq 0,
\end{align*} 
for any $\alpha$, such that $-c_2 \leq a_1\leq  \alpha \leq -c_1\leq a_2 $.  Hence there are infinitely many halfspace solutions corresponding to $ p^1$, and 
$ \varGamma_1$ can be any line on the plane (depending on $ \alpha$).
Next, assuming that  $ a_1+c_2>0$ and $ c_1+a_2>0$, we show that there are no halfplane solutions corresponding to the upper obstacle $ p^2$.
Indeed, in this case
\begin{align*}
\delta_2(\alpha)= -a_1c_1+a_2c_2-\alpha(a_1+c_2)+\alpha(a_2+c_1) \geq \\
 -a_1c_1+a_2c_2+c_1 (a_1+c_2)+a_1(a_2+c_1)
=(a_1+c_2)(a_2+c_1)>0,
 \end{align*}
and the statement follows from Lemma \ref{lemmahalfspace}, \eqref{delta2}. In this case  $\varGamma_2=\{0\}$  for any  halfspace solution $u^1$.

\par
Finally, if $ a_1+c_2>0$ but $c_1+a_2=0$, then $ \alpha = -c_1=a_2$, and we obtain only two 
halfspace solutions corresponding to $ p^2$.

\par
If $ a_1+c_2 <0$ and $ c_1+a_2 \leq 0$, then  we can consider the double obstacle problem with obstacles $-p^2 \leq -p^1$, and see that there at most two halfspace solutions
corresponding to $p^1$.

\end{proof}

\begin{corollary} \label{corehalfspace}
Under the assumptions of Theorems \ref{theorem1}, \ref{theoremhalfspace} and Corollary \ref{core}
we have that;
\begin{enumerate}
\item[Case 1)] If $ P \equiv 0$,
  there are infinitely many rotational invariant halfspace solutions corresponding to $p^1$ $( p^2 )$.
  \item[Case 2)] $P$ changes the sign.  There are infinitely many halfspace solutions corresponding to $p^1$ $ ( p^2 ) $.
    \item[Case 3)] $P$ has a sign, if $ P(x)\geq 0$, there are infinitely many rotational invariant halfspace solutions corresponding to $p^1$, and at most two
 halfspace solutions corresponding to $p^2$. Similarly, if $ P(x)\leq 0$, then are infinitely many rotational invariant halfspace solutions corresponding to $p^2$, and at most two halfspace solutions corresponding to $p^1$.
 \end{enumerate}
\end{corollary}

\bigskip

\section{Uniqueness of blow-ups, Case 1}

Let $ u$ be a solution to the double obstacle problem \eqref{equationsection3}, with polynomial obstacles
$p^1\leq p^2$, satisfying $p^1(x)=p^2(x)$ iff $x=0$.
We study the  uniqueness of blow-ups of $u$ in \textit{ Case 1,} i.e. when the polynomials $ p^i$
are given by 
\begin{equation*}
 p^1(x)=-ax_1^2-cx_2^2, ~~p^2(x)= cx_1^2+a x_2^2, \textrm{ where } a+c >0.
\end{equation*}

Consider the following harmonic polynomial $ h(x):=\frac{-a+c}{2}x_1^2+ \frac{a-c}{2}x_2^2$, then 
\begin{equation*}
 p^1(x)-h(x)= \frac{a+c}{2}(-x_1^2-x_2^2), ~ ~ p^2(x)-h(x)= \frac{a+c}{2}(x_1^2+x_2^2).
\end{equation*}
Thus it is enough to study the uniqueness of blow-ups in the case
\begin{equation} \label{p1=-p2}
 p^1(x)=-x_1^2-x_2^2, \textrm{ and } p^2(x) =x_1^2+ x_2^2.
\end{equation}
 From now on we study the solution $ \frac{2(u-h)}{a+c}$ instead of $u$, but still call it $u$.

\par
Let $ r_j \rightarrow 0+$, as $ j\rightarrow\infty $, and 
\begin{align*}
 u_0(x):=\lim_{j \rightarrow\infty} \frac{u(r_j  x)}{r_j ^2}
\end{align*}
be a blow-up of $ u $ at the origin. We know  that there exists
\begin{equation*}
\lim_{r\rightarrow 0} W(u , r, 0)= \lim_{j\rightarrow \infty} W(u_{r_j}, 1, 0 ) \equiv W(u_0, 1, 0).
\end{equation*}
Hence if $ \bar{u}_0$ is another blow-up solution, then $ W(u_0, 1, 0)=W(\bar{u}_0, 1, 0) $.
Denote by 
\begin{equation}\label{Ai}
  \mathcal{C}_i:= \{x=(x_1,x_2)\in \mathbb{R}^2;  u_0(x) = p^i(x)\},
  \end{equation}
  where $ p^i $ are the polynomials in 
\eqref{p1=-p2}.

\par
Let us 
calculate the values of $W (u_0, 1, 0) $ for all the possible blow-up solutions $u_0$.
By definition 
\begin{equation} \label{Weissu0}
\begin{aligned}
  W( u_0,1,0)= 
  \int_{B_1}  \lvert \nabla u_0 \rvert^2 dx  -2 \int_{\partial B_1} u_0^2 d S 
 + 2\int_{B_1}  \lambda_1 u_0 \chi_{\{ u_0 = p^1\}}+  \lambda_2 u_0 \chi_{ \{u_0= p^2\}} dx \\
 =- \int_{B_1}  u_0 \Delta u_0  dx 
+ 2\int_{B_1}  \lambda_1 u_0 \chi_{\{ u_0 = p^1\}}+  \lambda_2 u_0 \chi_{ \{u_0= p^2\}} dx \\
 = \lambda_1 \int_{B_1 \cap \mathcal{C}_1}  p^1  dx+ \lambda_2 \int_{B_1 \cap \mathcal{C}_2}  p^2  dx.
 \end{aligned}
\end{equation}
After substituting $ \lambda_1=-4$ and $ \lambda_2=4$ in \eqref{Weissu0}, we obtain
\begin{equation*}
\begin{aligned}
  W( u_0,1,0)= 
  4 \int_{B_1 \cap \mathcal{C}_1} r^3  dr d\theta+ 4 \int_{B_1 \cap \mathcal{C}_2}  r^3  drd\theta,
 \end{aligned}
\end{equation*}
and we may conclude from Theorem \ref{theorem1} that 
\begin{equation} \label{wlimit}
   W( u_0,1,0)= 
   \begin{cases}
  0,  & \mbox{if } u_0 \textrm{ is a  harmonic second order polynomial} \\
  2 \pi,  & \mbox{if } u_0 \equiv p^1 \textrm{ or } u_0  \equiv p^2 \\
  {\pi}, & \mbox{if } u_0 \textrm{ is a halfspace or a double-cone solution.}
 \end{cases}
\end{equation}
This gives three types of possible blow-ups at a fixed free boundary point.

\par
Denote by
\begin{equation} \label{uj}
 u^j(x):=\frac{u(r_j x )}{r_j^2}, 
\end{equation}
and assume that 
\begin{equation} \label{ujconv}
 u^j\rightarrow u_0  \textrm{ in } C^{1,\gamma}(B_1).
\end{equation}
If $ u_0$ is a polynomial or a halfspace solution to the double obstacle problem, then $u_0$
is a blow-up solution to a single obstacle problem. In this case the known techniques can be used to prove the uniqueness of blow-ups and to analyse the free boundary.
We will provide a rigorous argument for that in the end of this section.

\par
For now we focus on the case when $ u_0$ is a double-cone solution.
According to Theorem \ref{theorem1}, $ u_0$ can be described in terms of parameters  $-1<\alpha_1, \alpha_2 <1$.
Let $\alpha_i = \cos \phi_i$, for some
$ 0< \phi_1, \phi_2 <\pi$.  According to Lemma \ref{qS}, $ \beta_{i}=  \pm \sqrt{1-\alpha_i^2}= \pm \sin \phi_i$, 
$ m_{i}= \frac{\beta_{i}}{\cos\phi_i-1}=\mp \tan{\frac{\phi_i}{2}}$ and 
$ k_{i}= \frac{\beta_{i}}{\cos\phi_i +1}=\pm \cot{\frac{\phi_i}{2}}$. 
Referring to \eqref{q12}, we see that
\begin{align*}
 q^i(r,\theta)= x_1^2 \cos \phi_i -x_2^2\cos \phi_i\pm 2 x_1x_2 \sin \phi_i 
 =r^2( \cos \phi_i \cos^2\theta -\cos \phi_i \sin^2 \theta \pm \sin 2\theta \sin \phi_1) \\
 = r^2 (\cos \phi_i  \cos 2\theta \pm \sin 2\theta \sin \phi_i) =
 r^2 \cos( 2 \theta \mp \phi_i).
\end{align*}
Hence without loss of generality   $ u_0$ is the following function
\begin{equation} \label{u01}
  u_{0}=\mu= \mu_{\phi_{1}, \phi_2}(r, \theta):=
  \begin{cases}
  r^2,  & \mbox{if }~ -\phi_2 \leq 2 \theta \leq \phi_1 \\
  r^2 \cos(2\theta-\phi_1), & \mbox{if }~  \phi_1 \leq  2 \theta \leq \pi +\phi_1 \\
  r^2\cos(2\theta+\phi_2), & \mbox{if } ~ -\pi -\phi_2  \leq  2 \theta \leq  -\phi_2 \\
    -r^2,  & \mbox{otherwise. }  \\
 \end{cases}
\end{equation}

\par
For further analysis we need the following two easy lemmas.
\begin{lemma}\label{easylemma}
 Let $u$ and $ u_0 $ be two solutions (with different boundary conditions) to the double obstacle problem in $B_2 \subset \mathbb{R}^n$, with
 given
 obstacles $ \psi^1 \leq \psi^2$. 
 Then for any $\zeta \in C^2_0({B_2})$ the following inequality holds
 \begin{equation} \label{w12conv}
    \int_{B_2} ( \nabla u -\nabla u_0) \cdot \nabla \left( \zeta^2 (u-u_0) \right) dx \leq 0.
     \end{equation}
 Furthermore,
 \begin{equation}\label{w12}
  \lVert u-u_0 \rVert_{W^{1,2}( B_{1})} \leq C_n   \lVert u-u_0 \rVert_{L^{2}( B_{2})} ,
 \end{equation}
where $C_n$ is just a dimensional constant.
\end{lemma}
\begin{proof}
The proof is quite standard. Given a solution $u$ to the double obstacle problem in $ B_1$, then 
for any 
  $ \zeta \in C^2_0(B_{1})$, the function 
$u_t(x):= u+t \zeta^2 (u_0-u) $ is admissible for 
  $ t> 0$  small enough depending only on $\zeta$. Hence 
  \begin{align*}
  \int_{B_2} \lvert \nabla u \rvert^2 dx \leq   \int_{B_2} \lvert \nabla u_t\rvert^2 dx =
   \int_{B_2} \lvert \nabla u \rvert^2 dx +2t  \int_{B_2} \nabla u \cdot \nabla \left( \zeta^2 (u_0-u) \right) dx \\
   +t^2  \int_{B_2} \left|  \nabla \left( \zeta^2 (u_0-u) \right) \right| ^2 dx, 
  \end{align*}
 after dividing the last inequality by $ t>0$, and taking the limit as $ t$ goes to zero, we obtain
 \begin{equation} \label{uu0}
  0 \leq  \int_{B_2} \nabla u \cdot \nabla \left( \zeta^2 (u_0-u) \right) dx = 
  \int_{B_2} (u_0-u) \nabla u \cdot \nabla \zeta^2 dx+\int_{B_2} \zeta^2  \nabla u \cdot \nabla ( u_0-u)dx.
 \end{equation}
 Similarly, the function $u_0+t \zeta^2 (u -u_0) $ is admissible for the double obstacle problem, having
solution $ u_0$. 
Therefore
 \begin{equation} \label{u0u}
  0 \leq  \int_{B_2} \nabla u_0 \cdot \nabla \left( \zeta^2 (u-u_0) \right) dx = 
  \int_{B_2} (u-u_0) \nabla u_0 \cdot \nabla \zeta^2 dx+\int_{B_1} \zeta^2  \nabla u_0 \cdot \nabla ( u-u_0)dx.
 \end{equation}
 The inequalities \eqref{uu0} and \eqref{u0u} together imply the inequality \eqref{w12conv}, and we proceed to the proof of the second statement in our lemma.
 
\par
Choose $ \zeta \in C^2_0(B_{3/2})$, such that $ 0 \leq \zeta \leq 1$ and $ \zeta \equiv 1$ in $B_{1}$.
Combining the inequalities \eqref{uu0} and \eqref{u0u}, we obtain
\begin{align*}
 \int_{B_2} \zeta^2 \lvert \nabla u - \nabla u_0 \rvert^2 dx \leq 
 -2 \int_{B_2} \zeta (u-u_0)(\nabla u -\nabla u_0 ) \cdot \nabla \zeta dx \\
 \leq 
 2 \int_{B_2} \lvert \nabla \zeta \rvert^2 (u-u_0)^2 dx 
 +\frac{1}{2}\int_{B_2} \zeta^2\lvert \nabla u -\nabla u_0 \rvert^2 dx,
\end{align*}
where we used Young's inequality in the last step. Hence
\begin{equation*}
\int_{B_{1}} \lvert \nabla u -\nabla u_0 \rvert^2 dx \leq 
\int_{B_2} \zeta^2\lvert \nabla u -\nabla u_0 \rvert^2 dx \leq 4 \int_{B_2} \lvert \nabla \zeta \rvert^2 (u-u_0)^2 dx 
\leq C_n \Arrowvert u- u_0 \rVert^2_{L^2(B_2)},
\end{equation*}
where $ C_n $ is just a dimensional constant, depending only on $ \zeta$. The proof of the lemma is now complete.
\end{proof}

\begin{lemma} \label{easylemma2}
Let $ \{u^j\}$ and $\{ \mu^j\}$ be given sequences of solutions to the double obstacle problem in $ B_2$. 
Assume that $ \delta^j:=|| u^j -\mu^j||_{L^2(B_2)} \rightarrow 0 $, as $ j\rightarrow \infty$. Define 
 $ v^j:= \frac{u^j-\mu^j}{\delta^j}$. Then up to a subsequence
 \begin{equation} \label{convvj}
 v^j \rightharpoonup v^0 \textrm{ weakly in } W^{1,2}(B_{1}), \textrm{ and }
 v^j \rightarrow v^0  \textrm{ in } L^{2}(B_{1}).
\end{equation}
Furthermore, if $v^0 \Delta v^0=0$ in a weak sense, then
\begin{equation} \label{strong}
|| v ^j-v^0||_{W^{1,2}(B_{1/2})} \rightarrow 0.
\end{equation}
\end{lemma}
\begin{proof}
According to Lemma \ref{easylemma}, $ \Arrowvert v^j \rVert_{ W^{1,2}(B_{1})}\leq C_n$,
where $ C_n$ is a dimensional constant. Hence \eqref{convvj} follows from the weak compactness of the space 
$  W^{1,2}$ and from the Sobolev embedding theorem.

\par
We will obtain the strong convergence in \eqref{strong}, if we show that 
\begin{equation}
\lim_{j\rightarrow \infty} || \nabla v^j ||_{L^2({B_{1/2}})} = || \nabla v^0 ||_{L^2({B_{1/2}})}.
\end{equation}
According to Lemma \ref{easylemma}, \eqref{w12conv}, for any $\zeta \in C^2_0({B_1})$ and for any $j$, the following inequality holds
 \begin{equation} 
    \int_{B_2}  \nabla v^j  \cdot \nabla \left( \zeta^2 v^j \right) dx \leq 0.
     \end{equation}
Hence
\begin{equation}\label{limsup}
\begin{aligned}
\limsup_{j\rightarrow \infty} \int_{B_1} \zeta^2 | \nabla v^j |^2 \leq -\lim_{j\rightarrow \infty} \int_{B_1} v^j \nabla v^j \cdot \nabla (\zeta^2)=
-\int_{B_1} v^0 \nabla v^0 \cdot \nabla (\zeta^2) \\
=\int_{B_1}\zeta^2 |\nabla v^0|^2 -\int_{B_1} \nabla v^0 \cdot \nabla {(\zeta^2v^0)}
= \int_{B_1}\zeta^2|\nabla v^0|^2,
\end{aligned}
\end{equation}
where we used $v^0 \Delta v^0=0$ in the last step.
On the other hand we have that 
\begin{equation} \label{liminf}
\liminf_{j\rightarrow \infty} \int_{B_1} \zeta^2 | \nabla v^j |^2 \geq  \int_{B_1}\zeta^2|\nabla v^0|^2,
\end{equation} 
by the weak convergence $\zeta \nabla v^j  \rightharpoonup \zeta \nabla  v^0 $  in $L^2(B_1)$.
Choosing $ \zeta \in {C_0^2(B_1)}$ such that $ \zeta=1$ in $ B_{1/2}$, we obtain $ || \nabla v^j-\nabla v^0 ||_{L^2(B_{1/2})} \rightarrow 0$.

\end{proof}

\begin{remark}
The strong convergence $ v^j \rightarrow v^0$ in $W^{1,2}(B_{1/2})$ will not be used when proving the uniqueness of double-cone blow-up  limits, but we will need it when discussing the uniqueness of halfspace blow-ups. \par 
If  $ u^j \rightarrow u_0$, where $ u_0$ is a homogeneous global solution, then $v^0 \Delta v^0 =0$ always holds, although in the statement of Lemma \ref{easylemma} we preferred to assume rather than prove it. 
We will prove that $ v^0\Delta v^0=0$ in our later discussion. 
\end{remark}

\begin{definition} \label{minimal}
 Let $ u $ be a solution to the double obstacle problem. We say that $ u_0 = \mu_{\phi_{1},\phi_2}$ is a minimal 
 double-cone solution with respect to $ u$, if
 \begin{equation}
  \Arrowvert u-\mu_{\phi_{1},\phi_2} \rVert_{L^2 (B_1)} \leq 
  \Arrowvert u-\mu_{\phi_{1}+\tau,\phi_2+\delta} \rVert_{L^2 (B_1)},
 \end{equation}
for any $\tau, \delta$, such that $ \lvert \tau \rvert < \pi- \phi_1$ and 
$ \lvert \delta \rvert < \pi- \phi_2$.
\end{definition}

It follows from Definition \ref{minimal}, that if $ \mu_{\phi_{1},\phi_2}$ is a minimal double-cone solution
with respect to $u$, then 
\begin{equation}  \label{orthogmu}
 \int_{\phi_i/2}^{\pi/2+\phi_i/2}\int_0^1 \sin(\phi_i-2\theta )  \left( u(r,\theta )-\mu_{\phi_{1},\phi_2}(r,\theta)\right)r^3 dr 
 d \theta =0, \textrm{ for } i=1,2.
 \end{equation}
We derive equation \eqref{orthogmu}  by taking the partial derivatives of 
\mbox{ $ \Arrowvert u-\mu_{\phi_{1}+\tau,\phi_2+\delta} \rVert_{L^2 (B_1)}$ } at the origin 
 with respect to variables $ \tau$ and $ \delta$.

\begin{proposition} \label{vjvo}
Let $ u^j$ be a sequence of solutions to the double obstacle problem 
with obstacles $ p^1(x)=-x_1^2-x_2^2$ and $ p^1(x)=x_1^2+x_2^2 $
in $ \Omega \subset \mathbb{R}^2$, 
$ B_2 \subset \subset \Omega $. Assume that \eqref{ujconv} holds, where $u_0= \mu $ 
is given by \eqref{u01}. Denote by 
\begin{equation} \label{vj}
 v^j(x):= \frac{u^j(x)-\mu^j(x)}{ \Arrowvert u^j -\mu^j\rVert_{L^2(B_2)} },
\end{equation}
where $ \mu^j$ is a minimal double-cone solution with respect to $ u^j$.
Then \eqref{convvj} holds up to a subsequence, where
 $ v^0 \equiv 0$ in $ \mathcal{C}_1 \cup \mathcal{C}_2$, and $ \Delta v^0 =0$  in 
each of the components of the noncoincidence set $ \Omega_{12}= \mathcal{S}_1 \cup \mathcal{S}_2$, where the cones $ \mathcal{C}_i$ and $\mathcal{S}_i$ correspond to $\mu$.
Furthermore, it follows from the minimality assumption that
\begin{equation} \label{orthogonality} 
\int_0^1  \int_{\phi_i/2}^{\pi/2+\phi_i/2} v^0(r, \theta ) \sin(2\theta-\phi_i )d \theta r^3 dr=0 \textrm{ for  }
i= 1,2.
\end{equation}
\end{proposition}

\begin{proof}

\par
It follows from the minimality assumption  and from the triangle inequality, that $\mu^j\rightarrow \mu$;
\begin{equation*}
 \Arrowvert \mu^j-\mu \rVert_{L^2(B_1)} \leq  \Arrowvert u^j -\mu^j \rVert_{L^2(B_1)}  + \Arrowvert u^j -\mu \rVert_{L^2(B_1)} \leq 2  \Arrowvert u^j -\mu \rVert_{L^2(B_1)} 
 \rightarrow 0.
 \end{equation*}

\par
We  show that for any  $ K \subset \subset \mathcal{C}_1 \cap B_1 $, the functions $ v^j$ vanish
in $ K$ for large 
$ j$, since then we may conclude that $ v^0 \equiv 0$ in $ \mathcal{C}_1$. Let 
$ K\subset \subset V \subset \subset \mathcal{C}_1$, 
and $ d:=\dist(K, \partial V)$.
 It follows from \eqref{ujconv} that for any $ \varepsilon>0$ there exists $ j(\varepsilon)$, such that 
$ \lvert u^j(x)- p^1(x) \rvert \leq \varepsilon $, for any $ x \in K $, provided $ j\geq j(\varepsilon)$ is large enough,
depending only on $ \varepsilon$. Take $ 0<\varepsilon < \frac{d^2}{4}$. 
Let us denote by $ w^j := u^j- p^1$, then $ 0 \leq w^j \leq \varepsilon $
solves the following normalized obstacle problem with zero obstacle
\begin{equation*}
 \Delta w^j= -\lambda_1 \chi_{\{w^j >0\}}= 4\chi_{\{w^j >0\}} \textrm{ in } K.
 \end{equation*}
Fix $ x_0 \in K$, if $ w_j (x_0) >0$, then we can apply the maximum growth lemma (Lemma 5 in \cite{Caf98}) for the solution to the 
classical obstacle problem, and obtain
\begin{equation}
  \frac{d^2}{4} > \varepsilon \geq \sup_{B_d(x_0)} w^j \geq \frac{d^2}{2} ,
\end{equation}
which is not possible, therefore $ w^j (x_0)=0$. Hence we may conclude that for $ j>>1$ large, $ v^j$ is vanishing
 in $ K$, for any $ K \subset \subset \mathcal{C}_1\cap B_1$.
 
 \par
 Let $U \subset \subset  \mathcal{S}_1 \cap B_1$ be any open set, then $ p^1< u^j <p^2$
and  $ p^1< \mu^j <p^2$ for large $j$, hence $ \Delta v^j=0$, and after passing to the limit as $ j\rightarrow \infty$, we obtain 
$ \Delta v^0=0$ in $ \mathcal{S}_1 \cap B_1$. Hence $ v^0$ is a harmonic function outside its support, and $ v^j \rightarrow v^0 $ in $W^{1,2}(B_{1/2})$.

\par
We obtain \eqref{orthogonality} by passing to the limit as $ j\rightarrow \infty$ in equation \eqref{orthogmu} 
applied for the solutions $u^j$.

\end{proof}

\begin{lemma} \label{lemmav0}
 Let $ v^0 $ be the function in Proposition \ref{vjvo}, then 
 \begin{equation} \label{v0s}
  \lVert v^0(sx) \rVert_{L^2(B_{1})} \leq s^4  \lVert v^0 \rVert_{L^2(B_{1})},
 \end{equation}
for any $ 0<s<1$.
\end{lemma}

\begin{proof}
According to Proposition \ref{vjvo},  $ v^0$ is a harmonic function
in the sector
\begin{equation*}
\mathcal{S}_1 \cap B_1= \{ \phi_1 \leq 2\theta \leq \phi_1+ \pi \} \cap B_{1},
\end{equation*}
and $v^0$ satisfies
the following boundary 
conditions in the trace sense;
\begin{equation} \label{zerobd}
 v^0(r,\phi_1/2)= v^0(r, \phi_1/2+\pi/2)=0, \textrm{ for all } 0<r<1.
\end{equation}
Therefore
\begin{equation} \label{sum}
 v^0(r,\theta )= \sum_{k=1}^{\infty} r^{2k} \left( A_k \cos(2k\theta)+B_k \sin(2k\theta) \right)~ \textrm{ in }~
 \mathcal{S}_1 \cap B_1,
\end{equation}
and according to \eqref{zerobd}, we have that
\begin{equation} \label{bdc}
A_k \cos k\phi_1 + B_k \sin k\phi_1=0, \textrm{ for } k= 1,2,... .
\end{equation}
We claim that \eqref{bdc} implies 
\begin{equation} \label{kthterm}
 \int_{\phi_1/2}^{\pi/2+\phi_1/2} 
 \left( A_k \cos(2k\theta)+B_k \sin(2k\theta) \right)  \sin(2\theta-\phi_1 )d \theta =0, \textrm{ for all }
 k =2,3, ... .
\end{equation}
The proof of \eqref{kthterm} is straightforward. Fix $ k \geq 2$, and assume that $ \sin k\phi_1 \neq 0$, 
then 
\begin{equation*}
B_k= - A_k \frac{\cos k\phi_1}{\sin k\phi_1},
\end{equation*}
hence we obtain
\begin{align*}
 \int_{\phi_1/2}^{\pi/2+\phi_1/2} 
 \left( A_k \cos(2k\theta)+B_k \sin(2k\theta) \right)  \sin(2\theta-\phi_1 )d \theta \\ 
 =  \frac{A_k}{\sin k\phi_1} \int_{\phi_1/2}^{\pi/2+\phi_1/2} 
 \left(  \cos(2k\theta) \sin k\phi_1 -\cos k\phi_1  \sin(2k\theta) \right)  \sin(2\theta-\phi_1 )d \theta \\
 =\frac{-A_k}{\sin k\phi_1} \int_{\phi_1/2}^{\pi/2+\phi_1/2} 
\sin( k(2\theta-\phi_1)) \sin(2\theta-\phi_1 )d \theta = \frac{-A_k}{2 \sin k\phi_1} 
\int_0^\pi \sin kt \sin t dt =0, ~ \textrm{ if } k \neq 1.
\end{align*}
If $ \sin k\phi_1 = 0$, then $ \cos k\phi_1 \neq 0$, and the proof of \eqref{kthterm} works similarly.

\par
In the next step we show that $A_1= B_1=0$. 
By using the orthogonality property \eqref{orthogonality} and \eqref{kthterm}, and employing elementary 
trigonometric identities, we obtain
\begin{align*}
 0= \int_0^1  \int_{\phi_1/2}^{\pi/2+\phi_1/2} v^0(r, \theta ) \sin(2\theta-\phi_1 )d \theta r^3 dr \\
 = \int_0^1  \int_{\phi_1/2}^{\pi/2+\phi_1/2}\sum_{k=1}^{\infty} r^{2k+3}
 \left( A_k \cos(2k\theta)+B_k \sin(2k\theta) \right)  \sin(2\theta-\phi_1 )d \theta  dr \\
 = \sum_{k=1}^{\infty} \frac{1}{2(k+2)}\int_{\phi_1/2}^{\pi/2+\phi_1/2} 
 \left( A_k \cos(2k\theta)+B_k \sin(2k\theta) \right)  \sin(2\theta-\phi_1 )d \theta \\
 =\frac{1}{4} \int_{\phi_1/2}^{\pi/2+\phi_1/2} 
 \left( A_1 \cos(2\theta)+B_1 \sin(2\theta) \right)  \sin(2\theta-\phi_1 )d \theta \\
 =\frac{1}{4}\int_{\phi_1/2}^{\pi/2+\phi_1/2} 
 -A_1 \sin \phi_1 \cos^2 2\theta +B_1 \cos \phi_1 \sin^2 2 \theta d \theta =
 \frac{\pi}{16} ( -A_1 \sin \phi_1 +B_1 \cos \phi_1).
\end{align*}
Hence
\begin{equation} \label{a1b1}
 -A_1 \sin \phi_1 +B_1 \cos \phi_1=0.
\end{equation}
On the other hand
\begin{equation} \label{b1a1}
 A_1 \cos  \phi_1+ B_1 \sin \phi_1=0,
\end{equation}
which is \eqref{bdc} for $ k=1$. It is easy to see that  \eqref{a1b1} together with \eqref{b1a1} imply
$ A_1= B_1=0$. By \eqref{sum},
\begin{equation*}
 v^0(sx)=v^0(sr, \theta)=s^4 \sum_{k=2}^{+\infty}s^{2(k-2)}r^{2k}( A_k \cos(2k\theta)+B_k \sin(2k\theta))
 \textrm{  in } \mathcal{S}_1 \cap B_1.
\end{equation*}
Hence we obtain
\begin{equation} \label{S1}
  \lVert v^0(sx) \rVert_{L^2(\mathcal{S}_1 \cap B_{1})} \leq s^4  \lVert v^0 \rVert_{L^2(\mathcal{S}_1 \cap B_{1})}.
\end{equation}
Analogously, 
\begin{equation} \label{S2}
  \lVert v^0(sx) \rVert_{L^2(\mathcal{S}_2 \cap B_{1})} \leq s^4  \lVert v^0 \rVert_{L^2(\mathcal{S}_2 \cap B_{1})}.
\end{equation}
According to Proposition \ref{vjvo}, $ v_0 \equiv 0$ in $B_1 \setminus (\mathcal{S}_1 \cup \mathcal{S}_2)$, hence
\begin{equation*}
   \lVert v^0(sx) \rVert_{L^2( B_{1})} = \lVert v^0(sx) \rVert_{L^2(\mathcal{S}_1 \cap B_{1})}+ 
   \lVert v^0(sx) \rVert_{L^2(\mathcal{S}_2 \cap B_{1})}.
\end{equation*}
The desired inequality, \eqref{v0s}, now follows from \eqref{S1} and \eqref{S2}.
 
\end{proof}

\begin{corollary} \label{key}
 For any $ \varepsilon>0$ and $0< s<1 $, there exists $ \delta = \delta(\varepsilon, s )$ such that 
 if
 \begin{equation*}
 \Arrowvert u -\mu_0 \rVert_{L^2 (B_2)}\leq \delta,
 \end{equation*}
 then 
 \begin{equation}
\lVert u_s -\mu_0 \rVert _{L^2 (B_{1})} \leq 
 s \Arrowvert u -\mu_0 \rVert_{L^2 (B_{1})}  +\varepsilon  \Arrowvert u -\mu_0 \rVert_{L^2 (B_2)},
 \end{equation}
where $ \mu_0$ is a minimal double-cone solution with respect to $ u$.
\end{corollary}

\begin{proof}
We argue by contradiction. Let $u^j $ be a sequence of solutions to the double obstacle problem and assume that 
\begin{equation*}
 \Arrowvert u^j -\mu^j \rVert_{L^2 (B_1)}:= \delta_j \rightarrow 0, 
\end{equation*}
but there exist $ 0 <s  <1$ and $ \varepsilon> 0$, such that 
 \begin{equation} \label{assumption}
\lVert u^j_s -\mu^j \rVert _{L^2 (B_{1})} >  
 s \Arrowvert u^j -\mu^j \rVert_{L^2 (B_{1})}  +\varepsilon \delta_j.
 \end{equation}
Let $v^j$ be the sequence defined by \eqref{vj}, then by \eqref{assumption}
\begin{equation} \label{contr}
 \Arrowvert v^j_s \rVert_{L^2(B_{1})} > s \Arrowvert v^j \rVert_{L^2(B_{1})} \textrm{ and }
 \Arrowvert v^j\rVert_{L^2(B_{1})} > \varepsilon.
\end{equation}
Applying Proposition \ref{vjvo}, we may pass to the limit in \eqref{contr} as $j \rightarrow \infty$, and obtain
\begin{equation*} 
 \Arrowvert v^0_s \rVert_{L^2(B_{1})} \geq s \Arrowvert v^0 \rVert_{L^2(B_{1})} \textrm{ and }
 \Arrowvert v^0 \rVert_{L^2(B_{1})} \geq  \varepsilon,
\end{equation*}
hence 
\begin{equation*} 
 \Arrowvert v^0_s \rVert_{L^2(B_{1})} \geq s \Arrowvert v^0 \rVert_{L^2(B_{1})} >
 s^2 \Arrowvert v^0 \rVert_{L^2(B_{1})},
\end{equation*}
and we derive a contradiction to Lemma \ref{lemmav0}.

\end{proof}

Assume that $ u_0 =\mu$ given by \eqref{u01}
 is a blow-up for $ u$ at the origin, that is \eqref{uj} and \eqref{ujconv} hold for a sequence 
 $ r_j \rightarrow 0$. We want to show that the blow-up of $ u $ at the origin is unique; 
 \begin{equation*}
 u_r= \frac{ u(rx)}{r^2} \rightarrow u_0(x),~ \textrm{as } ~ r \rightarrow 0,
 \end{equation*}
and describe the rate of convergence $u_r \rightarrow u_0$ as $ r\rightarrow 0+$.

\begin{proposition} \label{important}
 Let $ u$ be the solution to the double obstacle problem in
 $ \Omega$, $ B_2 \subset\subset \Omega  \subset \mathbb{R}^2 $, with obstacles
 $ p^1(x)=-x_1^2-x_2^2$ and $ p^2(x)= x_1^2+ x_2^2$. Assume that $ \Arrowvert u-\mu \rVert_{L^2(B_2)}= \delta  $
 is small,
 where $ \mu$ is a double-cone solution, then there exists a double-cone solution $ u_0$,
 such that $ u_r \rightarrow u_0$.
Furthermore, for any $ 0< \gamma <1$,
\begin{equation} \label{estimate}
 \Arrowvert u_r - u_0 \rVert_{L^2(B_{1})} \leq C_n r^\gamma \Arrowvert u-\mu \rVert_{L^2(B_2)},
\end{equation}
 provided $ \delta > 0 $ is small depending on $ \gamma$.
\end{proposition}

\begin{proof}
Let $ \frac{1}{4}\leq s <  \frac{1}{2}$ be a fixed number, and $ \tau:=s^\gamma >s$.
We use an induction argument to show that for  $ \delta>0$ small enough
\begin{equation} \label{induction}
 \Arrowvert u_{s^{k+1}} -\mu^k \rVert_{L^2(B_1)} \leq  \tau^{k+1} \delta, \textrm{ and }
\Arrowvert \mu^{k+1} -\mu^k \rVert_{L^2(B_1)} \leq 2 \tau^{k+1} \delta, ~~k=0,1,2,3,...
\end{equation}
 where by definition $ \mu^k$  is a 
minimal double-cone solution with respect to $ u_{s^k}$.

\par
Let us show that \eqref{induction} is true for $ k=0$. First we observe that by 
the triangle inequality and
the minimality assumption,
\begin{equation} \label{trmin}
 \Arrowvert \mu-\mu^0  \rVert_{L^2(B_1)} \leq  \Arrowvert u-\mu  \rVert_{L^2(B_1)} +
  \Arrowvert u-\mu^0  \rVert_{L^2(B_1)} \leq 2 \Arrowvert u-\mu  \rVert_{L^2(B_1)}.
\end{equation}
Note that since $ \mu^k$  are homogeneous of degree two functions, the following relation 
is true
\begin{equation} \label{muk8}
 \Arrowvert \mu^{k+1}-\mu^k\rVert_{L^2(B_2)} = 8  \Arrowvert \mu^{k+1}-\mu^k\rVert_{L^2(B_1)}, \textrm{ for all }
 k=0, 1,2,... .
\end{equation}
Now let us proceed to the proof of \eqref{induction} for $k=0$.
 According to Corollary \ref{key} and \eqref{trmin},
\begin{equation*}
\begin{aligned}
 \lVert u_s -\mu_0 \rVert _{L^2 (B_{1})} \leq 
 s \Arrowvert u -\mu^0 \rVert_{L^2 (B_{1})}  +\varepsilon  \Arrowvert u -\mu^0 \rVert_{L^2 (B_2)} 
 \leq  s \Arrowvert u -\mu \rVert_{L^2 (B_{1})} \\
 +\varepsilon  \Arrowvert u -\mu \rVert_{L^2 (B_2)}
 +\varepsilon  \Arrowvert \mu -\mu^0 \rVert_{L^2 (B_2)}
 \leq  s \Arrowvert u -\mu \rVert_{L^2 (B_{1})} +\varepsilon  \Arrowvert u -\mu \rVert_{L^2 (B_2)} \\
 +16 \varepsilon  \Arrowvert u -\mu \rVert_{L^2 (B_1)} \leq 
 (s + 17\varepsilon) \Arrowvert u -\mu \rVert_{L^2 (B_2)} \leq \tau  \Arrowvert u -\mu \rVert_{L^2 (B_2)},
 \end{aligned}
\end{equation*}
where we take $0< \varepsilon < \frac{\tau- s}{17}$. Thus we obtain 
$\lVert u_s -\mu^0 \rVert _{L^2 (B_{1})} \leq \tau \Arrowvert u -\mu \rVert_{L^2 (B_2)} $, hence 
\begin{equation*}
 \rVert \mu^1- \mu^0\rVert_{L^2(B_1)} \leq \rVert u_s- \mu^1 \rVert_{L^2(B_1)}+
 \rVert u_s- \mu^0\rVert_{L^2(B_1)} \leq 2 \lVert u_s -\mu^0 \rVert _{L^2 (B_{1})} \leq 2 \tau
 \Arrowvert u -\mu \rVert_{L^2 (B_2)},
\end{equation*}
which completes the proof of \eqref{induction} for $k=0$.

\par
Let us assume \eqref{induction} holds up to and including $ k$, we will show that 
\eqref{induction} holds for $ k+1$.
First note that 
$\Arrowvert u_{s^{k+1}} -\mu^{k+1} \rVert_{L^2(B_2)} $ is small. 
Indeed, since $1/4 < s< 1/2$, we obtain 
\begin{equation} \label{k+1mu}
\begin{aligned}
 \Arrowvert u_{s^{k+1}} -\mu^{k} \rVert_{L^2(B_2)} =\frac{1}{s^2} 
 \Arrowvert u_{s^{k}} -\mu^{k} \rVert_{L^2(B_{2s})} \leq 16 \Arrowvert u_{s^{k}} -\mu^{k} \rVert_{L^2(B_{1})} \\
 \leq 
 16 \Arrowvert u_{s^{k}} -\mu^{k-1} \rVert_{L^2(B_{1})} \leq 16 \tau^{k} \delta \leq  16 \delta
 \end{aligned}
\end{equation}
by the induction assumption.
According to Corollary \ref{key} for any $ \varepsilon> 0$, we can choose $ 16 \delta > 0$ to be small
depending on $ \varepsilon $ and $s$, and obtain
\begin{equation*}
\begin{aligned}
 \Arrowvert u_{s^{k+2}} - \mu^{k+1}\rVert_{L^2(B_1)} \leq 
 s  \Arrowvert u_{s^{k+1}} - \mu^{k+1}\rVert_{L^2(B_1)}+ 
 \varepsilon \Arrowvert u_{s^{k+1}} - \mu^{k+1}\rVert_{L^2(B_2)} \\
 \leq 
  s  \Arrowvert u_{s^{k+1}} - \mu^{k}\rVert_{L^2(B_1)} + 
   \varepsilon \Arrowvert u_{s^{k+1}} - \mu^{k}\rVert_{L^2(B_2)} + 
  \varepsilon \Arrowvert \mu^{k+1} - \mu^{k}\rVert_{L^2(B_2)} \\
  \leq s  \Arrowvert u_{s^{k+1}} - \mu^{k}\rVert_{L^2(B_1)} +
   16 \varepsilon \Arrowvert u_{s^{k}} - \mu^{k-1}\rVert_{L^2(B_1)}+
   8 \varepsilon \Arrowvert \mu^{k+1} - \mu^{k}\rVert_{L^2(B_1)}, 
 \end{aligned}
\end{equation*}
where we used \eqref{k+1mu}  and  \eqref{muk8} in the last step. Recalling our induction assumption, we obtain
\begin{equation} \label{musk}
 \Arrowvert u_{s^{k+2}} - \mu^{k+1}\rVert_{L^2(B_1)} 
   \leq ( s \tau^{k+1}  +16 \varepsilon \tau^{k} + 8 \varepsilon \tau^{k+1} )\delta \leq \tau^{k+2} \delta, 
\end{equation}
by choosing $ \varepsilon < \frac{\tau(\tau-s)}{16+ 8 \tau}$.
 It follows from the triangle inequality and
the definition of minimal double-cone solutions  that
\begin{align*}
\Arrowvert \mu^{k+2} -\mu^{k+1} \rVert_{L^2(B_1)} \leq
\Arrowvert u_{s^{k+2}}-\mu^{k+2}\rVert_{L^2(B_1)}+\Arrowvert u_{s^{k+2}}-\mu ^{k+1}\rVert_{L^2(B_1)} \\
\leq
2 \Arrowvert u_{s^{k+2}}-\mu ^{k+1}\rVert_{L^2(B_1)}
\leq 2 \tau^{k+2} \delta.
\end{align*}
 The proof of the inequalities
\eqref{induction} is therefore complete.

\par
Now we are ready to show that $ \mu^k $ is a Cauchy sequence, and therefore converges.
For any $ m, k \in \mathbb{N}$
\begin{equation*}
 \Arrowvert \mu^{k+m}- \mu^k \rVert_{L^2(B_1) } \leq
 \sum_{l=k}^{k+m-1}  \Arrowvert \mu^{l+1}- \mu^l \rVert_{L^2(B_1) } \leq 
 \sum_{l=k}^{k+m-1} \tau^{l+1} \delta \leq \frac{\tau^{k+1}}{1-\tau} \delta \rightarrow 0,
\end{equation*}
independent of $ m$. Hence there exists $ u_0$, such that $ \mu^k \rightarrow u_0$, furthermore
\begin{equation} \label{mu0}
\Arrowvert \mu^k -u_0 \rVert_{L^2(B_1)} \leq \frac{\tau^{k+1}}{1-\tau} \delta.
\end{equation}
The inequalities \eqref{induction} and \eqref{mu0} together with the triangle inequality imply that
\begin{equation}
 \Arrowvert u_{s^k} -u_0 \rVert_{L^2(B_1)} \leq  2 \tau^k \delta.
\end{equation}
Finally let us observe that for any $0< r<1$ there exists a nonnegative integer $ k$ such that 
$ s^{k+1} \leq r <s^k$. Hence
\begin{equation}
 \Arrowvert u_r -u_0\rVert_{L^2(B_1)} \leq s^{-3} \Arrowvert u_{s^k} -u_0\rVert_{L^2(B_1)} 
 \leq 2 \cdotp 4 ^3 \tau ^k \delta  \leq 4^4 r^\gamma \delta,
\end{equation}
where $ \gamma= \frac{\ln \tau}{\ln s} <1 $.

\end{proof}

\begin{corollary} \label{case1uniqueness}
 Assume that $ \mu$ given by \eqref{u01}
 is a blow-up for $ u$ at the origin, that is \eqref{uj} and \eqref{ujconv} hold for a sequence 
 $ r_j \rightarrow 0$. Then the blow-up of $ u $ at the origin is unique; 
 \begin{equation*}
  \frac{ u(rx)}{r^2} \rightarrow \mu (x),~ \textrm{as } ~ r \rightarrow 0.
 \end{equation*}
\end{corollary}
\begin{proof}
 Since $ u_{r_j} \rightarrow \mu $ as $ j \rightarrow \infty$, for any $ \delta > 0  $ small
 we can find a small $ \rho > 0 $ such that $ \Arrowvert u_\rho - \mu  \rVert_{L^2(B_2)} \leq \delta$.
 Now we can apply Proposition \ref{important} for the function $ u_\rho$, and obtain 
 $ u_{\rho r} \rightarrow u_0$ as $ r \rightarrow 0+$. Hence $ u_r \rightarrow u_0$ and $ u_0 =\mu$.
\end{proof}

\begin{theorem} \label{theorem2}
Let $ u $ be the solution to the two-dimensional double obstacle problem with obstacles 
$ p^1=-x_1^2-x_2^2$ and $ p^2= x_1^2+ x_2^2$. Assume that $ \Arrowvert u -\mu \rVert_{L^2(B_2)} =\delta $ is 
sufficiently small, where $\mu$ is a double-cone solution. 
Then in a small ball $ B_{r_0}$ the free boundary consists of four $ C^{1, \gamma}$- graphs meeting at the origin, denoted
by $ \varGamma_1^+, \varGamma_1^-, \varGamma_2^+,\varGamma_2^-$.
Neither $ \varGamma_1 =\varGamma_1^+
\cup \varGamma_1^-$ nor $ \varGamma_2= \varGamma_2^+
\cup \varGamma_2^-$ has a normal at the origin. The curves
$ \varGamma_1^+ $ and $ \varGamma_2 ^+$ cross at a right angle, the same is true for
$ \varGamma_1^- $ and $ \varGamma_2 ^-$.
\end{theorem}

\begin{proof}
 The proof of the theorem is based on Proposition \ref{important} and on similar estimates obtained for the 
 classical obstacle problem in \cite{JA}. 
 According to  Proposition \ref{important} there exists a double-cone solution $u_0$ such that
 \eqref{estimate} holds. Moreover, applying Lemma \ref{easylemma} we obtain
 \begin{equation} \label{nice}
  \Arrowvert u_r -u_0 \rVert_{W^{1,2}(B_{1/2})} \leq C r^\gamma \Arrowvert u-\mu \rVert_{L^2(B_2)}.
 \end{equation}
 Without loss of generality we assume that $ u_0$ is given by \eqref{u01}; that is 
 \begin{equation*} 
  u_0= \mu_{\phi_{1}, \phi_2}(r, \theta):=
  \begin{cases}
  r^2,  & \mbox{if }~ -\phi_2 \leq 2 \theta \leq \phi_1 \\
  r^2 \cos(\phi_1-2\theta), & \mbox{if }~  \phi_1 \leq  2 \theta \leq \pi +\phi_1 \\
  r^2\cos(\phi_2+2\theta), & \mbox{if } ~ -\pi -\phi_2  \leq  2 \theta \leq  -\phi_2 \\
    -r^2,  & \mbox{otherwise }  \\
 \end{cases}
\end{equation*}

\begin{figure}
\begin{center}
 \begin{tikzpicture}[scale=1.2]
\draw[ultra thin,->] (-2.5,0) -- (2.5,0) node[anchor=west] {$x_1$} ;
\draw[ultra thin, ->] (0,-2.5) -- (0,2.5) node[anchor=south] {$x_2$} ;
\draw[red,thick,dashed]  (0,0) -- (2.5,0.5)node[above]{$\varGamma_2^+$};
\draw[blue,thick,dashed]  (0,0) -- (-0.5,2.5)node[left]{{$\varGamma_1^+$}} ;
\draw[red,thick,dashed]   (0,0) -- (2.2,-2.2)node[below]{{$ \varGamma_2^-$}};
\draw[blue,thick,dashed]   (0,0) -- (-2.2,-2.2)node[above]{$\varGamma_1^-$};
 \node at (2.1,1.3) {\small{{$\Delta u =0$}}};
 \node at (1.8,-0.4) {\small{$u=p^2$}};
 \node at (-1.5,0.3) {\small{$u= p^1 $} };
  \node at (0,-1.4) {\small{ $ \Delta u=0$} } ;
  \node at (1,2) {$\mathcal{S}_1$};
  \node at (0.3,-2) {$ \mathcal{S}_2$};
  \draw [thick,red] plot [smooth, tension=0.7] coordinates { (0,0) (0.6,0.2) (1,0.5) (2.5,0.4) };
  \draw [thick,red] plot [smooth, tension=0.7] coordinates { (0,0) (0.7,-0.6)(1,-1.2)(1.5,-1.4) (1.9,-2.1)(2.3,-2.2) };
  \draw [thick,blue] plot [smooth, tension=0.7] coordinates { (0,0) (-0.1,0.4) (-0.4,1) (-0.4,2.5) };
   \draw [thick,blue] plot [smooth, tension=0.7] coordinates { (0,0) (-0.45,-0.5) (-1,-0.8) (-1.5,-1.8) (-2.2,-2) };
   \draw[fill] (1,0.5) circle [radius=0.025];
   \node [left] at (1,0.5) {\small{$x_0$}};
   \draw [->] (1,0.5) -- (0.875,1.125)node[left]{\tiny{$ \nu(0)$}};
    \draw [ultra thin,dashed] (0.75,1.75) -- (1.25,-0.75);
    \draw [->] (1,0.5) -- (1.02,1.13)node[right]{\tiny{$ \nu(x_0)$}};
      \draw [->] (0,0) -- (-0.15,0.75)node[above]{\tiny{$ \nu(0)$}};
      \draw [black] (1,0.5) circle [radius=0.5];;
           \draw  (-0.035,0.175) -- (0.135,0.21) -- (0.175,0.035);
         \draw  (-0.13,-0.13) -- (0,-0.242) -- (0.13,-0.13);
    \draw  (1.02,0.36) -- (1.19,0.39) -- (1.22,0.228);
\end{tikzpicture}
\end{center}
\begin{center}
 \footnotesize{\textbf{ Figure 4.1:} 
 {The local behaviour of the free boundary, with obstacles touching at a single point in \textit{Case 1}.\\
 \small{Here $ \varGamma_i^\pm$ are the pieces of the free boundary for $ u$, while the dashed lines 
are the free boundary to the double-cone solution $u_0$.}}}
 \end{center}
 \end{figure}

 As before, we denote by $ \mathcal{C}_i=\{u_0= p^i\}$, and let $ \vartheta_i $ be the opening angle for $\mathcal{C}_i$, then  $0 < \vartheta_i <\pi $, $ \vartheta_1= (\phi_1+\phi_2)/2$
 and $\vartheta_2=\pi-\vartheta_1$.

 \par 
 We want to show the regularity of $ \varGamma_2=\partial \{u=p^2\}$ in a neighbourhood of the origin. 
 We perform the proof in two steps.
 
 \par
 \textit{ Step 1:} We show that $ \varGamma_2 \cap B_{r_0} \subset \overline{(Q \setminus K)} $, 
 for any open cones  $ Q$ and $ K $,
 having  a common  vertex at the origin, such that  $ K \subset \mathcal{C}_2 \subset \overline{Q}$ and 
 $ \partial K \cap \partial \mathcal{C}_2 =
 \partial Q \cap \partial \mathcal{C}_2 =\{ 0 \} $, where  $r_0 >0$ is small depending on $K, Q$.
 
 \par
 Let $ K \subset \mathcal{C}_2$ be a cone with a vertex at the origin, such that 
 $ \partial K \cap \partial \mathcal{C}_2 = \{0\} $. Fix  $ 0 <\varrho < 1/8$, and denote by $ V:=
 K \cap \{\varrho<x <1/2\}$, and $ \sigma:= \dist(V,\partial \mathcal{C}_2)$.
 First we will show that $ u_r (x) = p^2 (x)$ in $ V $
for small $ r> 0$.  
Take $0<  \varepsilon< \frac{\sigma^2}{4}$, then there exists $ r_\varepsilon=r_\sigma$, such that  
$\lvert u_r(x) - u_0(x) \rvert \leq \varepsilon $ if $ r \leq r_\varepsilon$.
Let $ \omega := p^2-u_r $, for a fixed $ r<r_\varepsilon$, then $ 0 \leq  \omega \leq \varepsilon $ solves the following 
normalised obstacle problem with zero obstacle,
\begin{equation} \label{lambda2k}
 \Delta \omega= \lambda_2 \chi_{\{ \omega > 0\}} \textrm{ in } \mathcal{C}_2.
\end{equation}
Fix $x_0 \in V$, if $ \omega(x_0)>0$, then we can apply the maximum growth lemma (Lemma 5 in \cite{Caf98})
for the solution to the obstacle problem, and obtain
\begin{equation}
 \frac{\sigma^2}{4} >\varepsilon \geq  \sup_{B_\sigma(x_0)} \omega \geq \frac{\sigma^2}{2},
\end{equation}
which is not possible, hence $ \omega(x_0)=0$.

\par
Thus we have shown that $ \frac{u(rx)}{r^2} = p^2 (x)$ for all  $ r < r_\varepsilon$
 and any $ x \in K $, such that $ \varrho < \lvert x \rvert < 1/2$.  Hence $ u(y)= p^2(y)$ if 
 $\varrho r  <\lvert y \rvert < \frac{r}{2}$ for all $ r < r_\varepsilon$, and 
 therefore $ u=p^2 $ in $ K \cap B_{r_0}$, $ r_0:=r_\varepsilon/2$.

 \par 
 Taking another open cone $ Q$, with a vertex at the origin, and such that $ \mathcal{C}_2 \subset Q $, 
 $ \partial Q \cap \partial \mathcal{C}_2 =\{ 0 \} $,
 we show that $ \Gamma_2 \cap B_r \subset \overline{Q} $ if $r $ is small. Let $ \varrho > 0$, then $ u_0 - p^2 < 0$
 in $ Q \setminus B_\varrho$, and therefore $ u_r -p^2 < 0 $ in $ Q \setminus B_\varrho$ for small $r>0$.
Hence $ u < p^2 $ in $ Q \cap B_r$ for a small fixed $ r >0 $, and $ \varGamma_2 \cap Q \cap B_r= \emptyset$.
 
 \par
Now we can write $ \varGamma_2= \varGamma_2^+ \cup \varGamma_2^- $, where $ \varGamma_2^+ \cap 
\varGamma_2^- =\{0\}$, and $ \varGamma_2^\pm $ are "squeezed" between ${K}$ and ${Q}$.

\par
\textit{ Step 2:}
We show that $\varGamma_2^+  $ is a $ C^{1,\alpha}$-graph up to the origin.
 Fix any $ x_0 \in \varGamma_2 \cap B_{r_0/2} $, and denote by $ d:=\lvert x_0 \rvert$. 
 Let $ d_0:= \frac{d}{2} \sin\vartheta_K $, where $\vartheta_K $ is the opening angle of $ K$.
 Since $ x_0 \notin K$ and $ x_0 \in Q$, we see that $ B_{d_0}(x_0) \cap \mathcal{C}_1= \emptyset$
 and $ B_{d_0}(x_0) \cap \mathcal{S}_2= \emptyset$, see Figure 4.1. Hence the function $ \omega:= p^2 -u $
 solves the following normalised obstacle problem
 \begin{equation} \label{omegasolution}
  \Delta \omega= \lambda_2 \chi_{\{\omega>0\}} \textrm{ in } B_{d_0}(x_0).
 \end{equation}
Now let us show that $p^2 -u_0$ is a halfspace solution for \eqref{omegasolution}.
Denote by $ \nu(0)$ the unit upward normal to the line 
$\{ \theta = \phi_1/2 \}$, as indicated in Figure 4.1;
\begin{equation*}
 \nu(0)=\left(-\sin\frac{\phi_1}{2}, \cos\frac{\phi_1}{2}\right).
\end{equation*}
The following is true,
\begin{equation} \label{nu0mu}
 u_0(x)= p^2(x)- 2 (\nu(0) \cdot x)_+^2 ~~\textrm{ if }~ u_0(x) > p^1(x), \textrm{ and } x \notin \mathcal{S}_2.
\end{equation}
Indeed, according to Lemma \ref{lemmaqs}, if $ x \in \mathcal{S}_1$, then
\begin{align*}
p^2(x)-u_0(x)= x_1^2+x_2^2 -x_1^2 \cos \phi_1 + x_2^2 \cos \phi_1-2x_1x_2\sin \phi_1 \\
=2x_1^2 \sin^2 \frac{\phi_1}{2}+2x_2^2 \cos^2 \frac{\phi_1}{2} -4x_1x_2 \sin\frac{\phi_1}{2}\cos\frac{\phi_1}{2} \\
=2 \left(-x_1 \sin \frac{\phi_1}{2}+x_2 \cos\frac{ \phi_1}{2} \right)^2= 2( \nu (0)\cdot x)^2.
\end{align*}
Since $ u_0= p^2$ in $ \mathcal{C}_2$, the proof of \eqref{nu0mu} is complete. 
Hence $p^2- u_0 $ is a halfspace solution for the 
 obstacle problem in $ B_{d_0}(x_0)$, depending on the direction $ \nu(0)$.  
Therefore we obtain
\begin{equation} \label{chain}
 \begin{aligned}
 \Arrowvert \nabla'_{\nu(0)} \omega \rVert_{L^2(B_{d_0}(x_0))} =
 \Arrowvert \nabla'_{\nu(0)} \left( u-u_0-2(\nu(0)\cdot x)_+^2 \right) \rVert_{L^2(B_{d_0}(x_0))} \\
 = \Arrowvert \nabla'_{\nu(0)} (u-u_0) \rVert_{L^2(B_{d_0}(x_0))} \leq 
 2\Arrowvert \nabla (u-u_0) \rVert_{L^2(B_{d_0}(x_0))}
\end{aligned}
\end{equation}
where by definition $ \nabla_e':=\nabla- e(\nabla \cdot e)$ for a unit vector $e$.

\par
According to Lemma \ref{easylemma}
\begin{equation*}
   \Arrowvert u-u_0\rVert_{W^{1,2}(B_{d_0}(x_0))} \leq c \Arrowvert u-u_0\rVert_{L^2(B_2)},
\end{equation*}
hence by \eqref{chain}
\begin{equation*}
 \Arrowvert \nabla'_{\nu(0)} \omega \rVert_{L^2(B_{d_0}(x_0))} =
  \Arrowvert \nabla (u-u_0) \rVert_{L^2(B_{d_0}(x_0))} \leq c \Arrowvert u-u_0\rVert_{L^2(B_2)}
 \leq c\delta,
\end{equation*}
which says that $\omega$ is almost flat in the direction $\nu(0)$.
According to Theorem 8.1 in \cite{JA}, $ \varGamma_2 \cap B_{d_0/2}(x_0)$ is a $C^{1,\gamma}$- graph, 
and there exists a unit normal vector to $ \varGamma_2$ at the point
 $ x_0$, denote it by $ \nu(x_0)$. Furthermore, it follows from Corollary 8.1 in \cite{JA} and 
 inequality \eqref{nice} that 
 \begin{equation*}
 \begin{aligned}
  \lvert \nu(x_0)-\nu(0)\rvert \leq c d_0^{-1} \lVert \nabla'_{\nu(0)} \omega \rVert_{L^2(B_{d_0}(x_0))} 
  \leq c d_0^{-1}\Arrowvert \nabla (u-u_0)\rVert_{L^2(B_{2d})} \\
  =16c d_0^{-1} d  \left( \int_{B_{1/2}} \lvert \nabla u(4dy)-\nabla u_0(4dy)\rvert^2 dy  \right)^\frac{1}{2}
  \leq c d_0^{-1} d  \Arrowvert u_{4d}-u_0 \rVert_{L^2 (B_1)} \\
  \leq \frac{cd}{d_0} d^{\gamma} \Arrowvert u-u_0\rVert_{L^2(B_2)},
  \end{aligned}
 \end{equation*}
where $c$ stands for a general constant, and it does not depend on $d$.
Now we may conclude that
\begin{equation} \label{uniformc1alpha}
\lvert \nu(x_0) -\nu(0 )\rvert \leq \frac{cd}{d_0} d^\gamma \delta= 
\frac{c\lvert x_0\rvert^\gamma}{ \sin{\vartheta_K}}\Arrowvert u-u_0\rVert_{L^2(B_2)} ,
\end{equation}
and therefore $ \varGamma_2^+$ is a $C^{1,\gamma}$-graph up to the origin, for any $ 0<\gamma <1$.

\par
The proof of $ C^{1,\gamma}$-regularity for $ \varGamma_2^- $
can be obtained similarly. In that case $ \varGamma_2^- $ is almost flat in the direction 
$\nu_2^-(0)=\left(-\cos\frac{\phi_2}{2}, -\sin \frac{\phi_2}{2}\right) \neq \nu(0)$, and we see that 
 $\varGamma_2$ is Lipschitz, but it is not $ C^{1}$ at the origin.

 \par 
By a similar argument we can study $ \varGamma_1$, and see that 
$ \varGamma_1^+ $ is almost flat in the direction 
$ \nu_1^+(0)= \left( \cos \frac{\phi_1}{2}, \sin \frac{\phi_1}{2}\right)$.
  Observe that $ \nu_1^+(0)$ and $\nu(0)$ are orthogonal, which means 
  $ \varGamma_1^+ $ and $ \varGamma_2^+ $ cross at the origin.
 
\end{proof}

It  follows from energy characterisation of free boundary points, \eqref{wlimit},  and from  Corollary \ref{case1uniqueness} that if at a free boundary point $ x_0$, the solution $ u$ has a halfspace blow-up
solution corresponding to $ p^i$, then all possible blow-ups at $ x_0$ are also halfspace solutions.

\begin{theorem} \label{remarkhalfspace}
Let $u$ solve the double obstacle problem in $ B_1$ with obstacles $p^1\leq p^2$.
If $ u_{r_j}\rightarrow u^1 $ as $ j \rightarrow \infty$, where $ u^1$ is  a halfspace solution corresponding to $ p^1$, then the blow-up of $u$ at the origin is unique, and $ \varGamma_1 \cap B_{1/2}$ 
is a $ C^{1,\gamma}$-curve.
\end{theorem}

\begin{proof}
Without loss of generality we may assume that $ u^1=p^1$ if $ x_2\leq 0$. 
We can employ the same idea we used in Definiton \ref{minimal}, when defining minimal double-cone solutions, in order to define 
minimal halfspace solutions corresponding to $p^1$: We say that $ u^1$ is a minimal halfspace solution with respect to $u$, if it is closer to $ u$ in $L^2(B_1)$-norm than any other halfspace solution. 
Employing a very similar flatness improvement argument, 
we obtain  the uniqueness of blow-up limits, and the $ C^{1,\gamma}$-regularity of $ \varGamma_1$.  
Let us provide a  brief sketch of the proof. 

\par
Let $ \mu^j $ be a minimal halfspace solution with respect to $ u^j:=u_{r_j}$, and let $ v^j$ be the function defined in $ \eqref{vj}$, then 
$ v^j \rightarrow v^0$ weakly in $ W^{1,2}(B_1)$ and strongly in $L^2(B_1)$, where $ v^0$ is harmonic in the halfplane $\{x_2 >0\}$, and  vanishes on the halfplane $ \{x_2 \leq 0\}$, see the 
proof of Proposition \ref{vjvo}. Hence $v^0 \Delta v^0=0$ in a weak sense, and 
\begin{equation}
v^0(x)=v^0(r,\theta)=\sum_{k=1}^{\infty} A_k r^k \sin(k\theta) \textrm{ in }  \{x_2 >0\}.
\end{equation}
By Lemma \ref{easylemma2}, $ v^j \rightarrow v^0$ in $W^{1,2}(B_{1/2})$, and therefore $ A_1=0$.
Furthermore, by  the minimality assumption, $ v^0$ is orthogonal to the function $r^2 \sin{(2\theta)}$, and therefore $A_2=0$, which implies that
\begin{equation*}
  || v^0(sx)||_{L^2(B_1)}\leq s^3 ||v^0||_{L^2(B_1)},
  \end{equation*}
we refer to the proof of Lemma \ref{lemmav0} for a similar argument. Repeating the iteration argument in Corollary \ref{key} and Proposition \ref{important}, we will obtain that 
$\varGamma_1$ is a $C^{1,\gamma}$-curve in a neighbourhood of the origin.
We leave out the details, since the technique is very similar to the argument we used earlier.
\end{proof}

\par 
Let us observe that Theorem \ref{remarkhalfspace} on the uniqueness of halfspace blow-up limits holds  also in \textit{Cases 2,~3}, since we do not use any relation between
obstacles in the proof. All we need is to know that we can rotate the free boundary, and choose a minimal halfspace solution, which provides the precious orthogonality property.

\par 
 The following lemma from one variable calculus will be quite useful when showing the uniqueness of blow-ups in case there are only finitely many blow-up solutions. 
\begin{lemma} \label{small}
Let $ f$ be a nonnegative continuous function in the interval $ (0,1)$, and  assume that 
\begin{equation}
\liminf_{t \rightarrow 0+} f (t)=0 \textrm{  and }  \limsup_{t\rightarrow 0+} f (t)=A >0. 
\end{equation}
Then for any $ 0<a <A$ there exists a sequence $ t_j \rightarrow 0 $, such that $ f(t_j) \rightarrow a$.
\end{lemma}
\begin{proof}
Fix $ 0<a <A$, and let $ y_j \rightarrow 0+ $ and $ z_j \rightarrow 0+ $ be such that $ \vert f( y_j ) -A \vert \leq \frac{1}{j} $ and $ 0 \leq f(z_j ) \leq \frac{1}{j}$. Taking $ j$ large enough, we insure that $ f(y_j ) >a $ and $ f(z_j ) <a$.  Since 
$   f$ is a continuous function, by intermediate value theorem there exists $ t_j \rightarrow 0+ $, such that $f(t_j)=a$.

\end{proof}

\begin{corollary}
Let $ u$ be the solution to the double obstacle problem with polynomial obstacles $ p^1 \leq p^2$. 
If $ u_{r_j} \rightarrow p^i$ through a subsequence, then the blow-up of $ u$ at the origin is unique.
\end{corollary}

\begin{proof}
Let $ f (r):= \Arrowvert u_r - p^1 \rVert_{L^2(B_1)}$.  If there exists a sequence $ s_j \rightarrow 0+$, such that $ u_{s_j} \rightarrow p^2$. 
Then the function $f $ satisfies the assumptions in Lemma \ref{small}, but $f$ has only two limit points.
Therefore there exists $ \lim_{r \rightarrow 0+}f(r)=0$.
\end{proof}

\bigskip

\section{Uniqueness of blow-ups, Case 2}

The uniqueness of blow-ups in \textit{Case 2)} follows from Theorem \ref{theorem1} and from Lemma \ref{small}.

\begin{theorem} \label{case2uniqueness}
Let $ u $ be a solution to the double obstacle problem  with obstacles 
\begin{equation*}
  p^1( x) = a_1 x_1^2+ c_1 x_2^2~\textrm{ and } ~ p^2( x) = a_2 x_1^2 + c_2 x_2^2,
\end{equation*}
satisfying \eqref{lambda12} and \eqref{a2a1} and assume that $ (a_1+c_2)(c_1+ a_2) < 0$.  If
$ u_{r_j}\rightarrow \mu$ for a  subsequence $ r_j \rightarrow 0+$, where $ \mu $ is a double-cone solution, then 
$u $ has a unique blow-up at the origin;
\begin{equation}
\frac{u(rx) }{r^2} \rightarrow \mu(x).
\end{equation}

\end{theorem}

\begin{proof}
The proof follows from Lemma \ref{small}. Assume that $ u_{r_j} \rightarrow \mu$ for a  subsequence $ r_j \rightarrow 0+$, and denote by
$ f (r):= \Arrowvert u_r - \mu \rVert_{L^2(B_1)}$.  If there exists a sequence $ s_j \rightarrow 0+$, such that $ u_{s_j} \rightarrow \mu_1$, and $ \mu_1  \neq \mu$. 
Then the function $f $ satisfies the assumptions in Lemma \ref{small}, but according to Theorem \ref{theorem1}, $f$ has only four limit points.
Therefore there exists $ \lim_{r \rightarrow 0+}f(r)=0$.
\end{proof}

Now let us discuss the speed of the convergence $   \lVert u_r-\mu \rVert_{L^2(B_{1})} \rightarrow 0$ as $r\rightarrow 0$, and the regularity of the free boundary.

\begin{remark}
Let $u $ be a solution to the double obstacle problem in \textit{Case 2}, and assume that  $u_{r_j} \rightarrow \mu$, where $\mu$ is a double-cone solution.
 Arguing as we did in the  proof of Theorem \ref{theorem2}, and employing Theorem \ref{case2uniqueness}, we can show  that the free boundary is a 
union of four $ C^{1}$-curves. Although we do not have a uniform estimate  like \eqref{uniformc1alpha} yet.
\end{remark}

With a modification of our flatness improvement argument, we can show uniform $ C^{1,\gamma}$-regularity of $\varGamma_i$ up to the origin also in \textit{Case 2}.

\begin{theorem}
Let $ u $ be a solution to the double obstacle problem  in $B_1 $ with obstacles 
  $p^i( x) = a_i x_1^2+ c_i x_2^2$,
satisfying \eqref{lambda12} and \eqref{a2a1} and assume that $ (a_1+c_2)(c_1+ a_2) < 0$. 
Let  $ u_r \rightarrow \mu$, where $\mu$ is a double-cone solution. Denote by $ \vartheta $ the opening angle of $ \mathcal{S}_i$. Then
\begin{equation} \label{case2uniform}
  \lVert u_r-\mu \rVert_{L^2(B_{1})} \leq  C r^\gamma \lVert  u-\mu \rVert_{L^2(B_{2})},
  \end{equation}
where $ 0< \gamma <1$ depends on $ \vartheta$. Furthermore, the free boundary $ \varGamma_u$ is a union of four $ C^{1,\gamma}$-curves.
\end{theorem}
\begin{proof}
We provide only a brief sketch of the proof, since the detailed proof would be quite long, and  very similar to the proofs of \mbox{Proposition \ref{important}} and of Theorem \ref{theorem2}. Therefore we focus on the main differences of double-cone solutions in  \textit{Cases 1} and \textit{2}.
Let 
\begin{equation}
v_r :=\frac{u_r-\mu }{  \lVert u_r-\mu \rVert_{L^2(B_{2})}} ,
\end{equation}
where $\mu$ is a double-cone solution.
Then $ v_r \rightarrow v^0$ through a subsequence $ r_j \rightarrow 0+$ weakly in $W^{1,2}(B_1)$ and strongly in $L^2(B_1)$. Here $ \Delta v^0=0$ in $ \mathcal{S}_i$, and 
$ v^0=0$ in $ \mathbb{R}^2\setminus (\mathcal{S}_1 \cup \mathcal{S}_2)$, see the proof of Proposition \ref{vjvo}.
Hence
\begin{equation} \label{case2voa1b1}
v^ 0(r,\theta)=\sum_{k=1}^\infty r^{\alpha_k} \left(  A^i_k \cos{\alpha_k \theta} + B^i_k \sin{\alpha_k \theta}\right) \textrm{ in } \mathcal{S}_i, ~ i=1,2,
\end{equation}
where $\alpha_k =\frac{\pi k}{ \vartheta }$ and $\vartheta$ is the opening angle of the cone $\mathcal{S}_i$ for $i=1,2$.
According to Theorem \ref{theorem1}, $0<\cos^2{\vartheta} <1$, hence there are two possible cases; either $i)~0< \vartheta <\pi/2$ or  $ii) ~\pi/2< \vartheta <\pi$.
We discuss these two  cases separately.
\par
$i)~0< \vartheta <\pi/2$, then  $ \alpha_k= \pi k/\vartheta > 2k$, for all $k=1,2,....$.
It follows that 
\begin{equation} \label{kappa2}
   \lVert v_s^0 \rVert_{L^2(B_1)} \leq s^\kappa \lVert  v_0 \rVert_{L^2(B_1)},
   \end{equation}
     where $ \kappa = \pi/\vartheta -2$, and a standard iteration argument leads to \eqref{case2uniform} with 
$0< \gamma<\kappa$.

\par
$ii) ~\pi/2< \vartheta <\pi$, then $ 1<\alpha_1=\pi/ \vartheta <2$ and $\alpha_2=2\pi/\vartheta   >2$. We will obtain 
\eqref{kappa2}
    with $ \kappa = 2\pi/\vartheta -2$, if we show that $ A^i_1=B^i_1=0$ in \eqref{case2voa1b1}.
Assume not, i.e. if  $ |A^1_1|+|B^1_1|>0$, then $ ||v^0_s||_{L^2(B_1)} \geq c s^{\alpha_1-2}$, where $c>0$ is a fixed constant depending on $A^1_1,B^1_1$ and $0 <s<1$ is any number. Since $ v^j:=v_{r_j}  \rightarrow v^0$ in $ L^2(B_1)$,
then for any fixed $ 1>\tau >s>0$ there exist $ \varepsilon >0$ small, such that
\begin{equation} \label{usvscase2}
 \lVert u_s-\mu \rVert_{L^2(B_{1})} \geq c \tau^{\alpha_1-2}  \lVert u-\mu \rVert_{L^2(B_{2})}, 
\end{equation}
 provided $  \lVert u-\mu \rVert_{L^2(B_{2})} \leq \varepsilon$. Indeed, if the last statement is not true, then there exist $ 1>\tau >s>0$, 
  and a sequence  $ u^j$, such that $   \lVert u^j -\mu \rVert_{L^2(B_{2})}\leq \varepsilon_j \rightarrow 0+$, but 
 \begin{equation} \label{ujsvscase2}
\frac{ \lVert u^j_s-\mu \rVert_{L^2(B_{1})} }{  \lVert u^j-\mu \rVert_{L^2(B_{2})}}< c \tau^{\alpha_1-2} .
\end{equation}
After passing to the limit as $ j \rightarrow \infty$ in \eqref{ujsvscase2}, we obtain $ ||v^0_s||_{L^2(B_1)} \leq c \tau^{\alpha_1-2}$, contradicting 
$ ||v^0_s||_{L^2(B_1)} \geq c s^{\alpha_1-2}$.

 \par
 According to  Theorem \ref{case2uniqueness}  for any $ \varepsilon >0$ small there exists $k_\varepsilon$ such that 
 $  \lVert u_{s^k}-\mu \rVert_{L^2(B_{2})}\leq \varepsilon $ for all $ k \geq k_\varepsilon$.
 The latter together with \eqref{usvscase2} implies that for any positive integer $ m$
 \begin{align*}
 \varepsilon \geq  \lVert u_{s^{k+m}}-\mu \rVert_{L^2(B_{2})} \geq c\tau^{\alpha_1-2} \lVert u_{s^{k+m-1}}-\mu \rVert_{L^2(B_{2})} \\
 \geq (\tau^{\alpha_1-2}c)^2  \lVert u_{s^{k+m-2}}-\mu \rVert_{L^2(B_{2})}\geq...\geq (\tau^{\alpha_1-2}c)^m  \lVert u_{s^{k_\varepsilon}}-\mu \rVert_{L^2(B_{2})}.
 \end{align*}
 Recalling that $ \alpha_1 <2$, and taking $1 >\tau >s>0$ small, we obtain a contradiction, when letting $ m \rightarrow \infty$. Hence $ |A^i_1|=|B^i_1|=0$,
 and \eqref{kappa2} holds with $ \kappa= \alpha_2-2$. An iteration argument similar to the one used in Corollary \ref{key} and Proposition \ref{important} leads to the following inequalities,
 $ \lVert u_{s^k}  -\mu \rVert_{L^2(B_{1})} \leq \tau^{k \gamma} \lVert u-\mu   \rVert_{L^2(B_{2})} $, whith $1>\tau >s$ and 
 $ \lVert u_{r} -\mu  \rVert_{L^2(B_{1})} \leq C r ^{ \gamma} \lVert u-\mu   \rVert_{L^2(B_{2})}$, where $ 0<\gamma < \alpha_2-2$, and $ 0<r<1$ small.  Applying the proof of Theorem \ref{theorem2}, we deduce  that $ \varGamma_i$ is Lipschitz and consists of two $C^{1,\gamma}$-curves, meeting at the origin.

 \end{proof}

\par
Now let us briefly  discuss the uniqueness of halfspace blow-up limits. 
Although in \textit{Case 2}, the lines $\varGamma_i$ are not rotationally invariant on the plane, but they are rotationally invariant inside a fixed cone, depending on the given obstacles.
Hence we can define minimal halfspace solutions in this case as well, and obtain the uniqueness of blow-ups (see Theorem \ref{remarkhalfspace}).

\bigskip

\section{Uniqueness of blow-ups, Case 3}

Consider the double obstacle problem with obstacles
\begin{equation*}
p^1(x)=a_1x_1^2+c_1x_2^2  ~\textrm{ and }  ~p^2(x)=a_2x_1^2+c_2x_2^2,
\end{equation*}
satisfying \eqref{a2a1}, \eqref{lambda12},
and let $ P $ be the polynomial in \eqref{p1+p2}.

\par
In \textit{ Case 3}, the polynomial $ P$ has a sign, and according to Theorems \ref{theorem1} and \ref{theoremhalfspace} we have only halfspace solutions.
 Without loss of generality we may assume that $ P \geq 0$. According to Theorem \ref{theoremhalfspace} there are infinitely many rotational invariant halfspace solutions corresponding to $ p^1$, and we can apply our flatness improvement argument (Theorem \ref{remarkhalfspace}) in order to show the uniqueness of blow-ups and $ C^{1,\gamma}$-regularity of $\varGamma_1$.

\bigskip

\section{An example of a double-cone solution in $\mathbb{R}^3$ }

Let $ u_0$ be a homogeneous global solution to the double obstacle problem with obstacles $ p^1\leq p^2 $ in $ \mathbb{R}^3$.
As usually we assume that the origin is a free boundary point, $p^1(0)=p^2(0)=0$, and we want to understand the behaviour  of the free boundary at the origin.
We split the discussion into three cases.

\par
If  $ p^1=p^2$ on a plane, then we obtain only halfspace solutions.
If $ p^1=p^2$ on a line, then we can analyse the possible blow-up solutions, based on our results 
obtained in dimension  $n=2$. In particular, we can see that in this case there are no three-dimensional double-cone solutions.
The proofs of the last statements can be obtained via a dimension reduction technique. However, we omit the proofs, since our aim is to find a three-dimensional double-cone solution.

\par
Our knowledge on the existence of double-cone solutions in dimension two suggest that we may obtain three-dimensional double-cone solutions, assuming  that $ p^1$ and $ p^2$ meet only at a single point. 
However, since in dimension $ n=3$ homogeneous degree two harmonic functions are 
not necessarily polynomials, the analysis is much more complicated.
In this section we give an example of a three-dimensional  double-cone solution, symmetric with respect to the $z$-axes and the $(x,y)$-plane.

\subsection{Solutions symmetric with respect to the $z$-axes}

Let $ p^1 \leq p^2 $ be given homogeneous degree two polynomials, meeting only at the origin.
We are looking for two closed cones 
$\mathcal{C}_1$, $\mathcal{C}_2$ and for a harmonic homogeneous degree two function $q$ in 
$ \mathbb{R}^3\setminus {(\mathcal{C}_1 \cup \mathcal{C}_2)}$, such that 
\begin{equation} \label{3Ddc}
\Delta q =0, ~p^1 <q <p^2 \textrm{ in } \mathbb{R}^3\setminus {(\mathcal{C}_1 \cup \mathcal{C}_2)}, 
\end{equation}
and
\begin{equation} \label{3Dbc}
q-p^1= |\nabla q-\nabla p^1|=0 \textrm{ on } \partial \mathcal{C}_1,
\textrm{ and } q-p^2= |\nabla q-\nabla p^2|=0 \textrm{ on } \partial \mathcal{C}_2.
\end{equation}

\par
 Let $ (r, \phi, \theta )$ represent the spherical coordinates in $ \mathbb{R}^3$,
\begin{equation*}
x=r \cos\phi \sin \theta, ~y= r\sin \phi \sin \theta, ~z=r\cos \theta, \textrm{ where } r \geq0, 0 \leq \phi <2\pi, 0 \leq \theta \leq \pi ,
\end{equation*}
then $q(r,\phi,\theta)= r^2 \zeta(\phi, \theta)$ by homogenuity. Furthermore, assume that $ q$ is  symmetric with respect to the $ z$-axes, i.e. $ \zeta(\phi, \theta)= \zeta(\theta)$.
Using the expression for the Laplace operator in spherical coordinates, we obtain the 
following ordinary differential equation for $ \zeta$,
\begin{equation} \label{Aode}
\zeta^{\prime \prime}+\frac{\cos\theta}{\sin \theta}\zeta^{\prime}+6 \zeta=0.
\end{equation}
We can see via a substitution that $ \zeta_1= 1+3\cos2\theta$ is a solution to \eqref{Aode}. Using reduction of order, we get another solution, 
$ \zeta_2 = 3 \cos\theta+\frac{1+3\cos 2 \theta}{4} \ln \frac{1-\cos\theta}{1+\cos\theta} $.
Thus the general solution to \eqref{Aode} is given by 
\begin{equation*}
\zeta= A( 1+3\cos2\theta)+B \left(    3 \cos\theta+\frac{1+3\cos 2 \theta}{4} \ln \frac{1-\cos\theta}{1+\cos\theta}  \right),
\end{equation*}
where $ A, ~B$ are real numbers.
Hence
\begin{equation}
q= Ar^2( 1+3\cos2\theta)+Br^2 \left(    3 \cos\theta+\frac{1+3\cos 2 \theta}{4} \ln \frac{1-\cos\theta}{1+\cos\theta}  \right).
\end{equation}

\par
We are looking for a solution to the double obstacle problem, symmetric with respect to the $ z$-axes. Hence we assume that the obstacles $ p^1$ and $ p^2$ are so, 
i. e. they do not depend on $ \phi$.
Let $ t:=\cos\theta $, $ t \in (-1,1)$, and take  
\begin{equation} \label{3Dpol}
p^1= (a_1+b_1t^2)r^2,  ~~ p^2=(a_2+b_2t^2)r^2, ~ \textrm{ where }  a_1<a_2,  ~ a_2-a_1+b_2-b_1>0,
\end{equation}
then $p^1\leq p^2$,   $p^1$ and $ p^2$ meet only at the origin. 
Also, observe that 
\begin{equation}
q=Ar^2(3t^2-1)+Br^2 \left( 3t+\frac{3t^2-1}{2}\ln \frac{1-t}{1+t}\right),
\end{equation}
for some $A, B$.
In order to construct an example of a double-cone solution, we want to find  $ A, B, a_i, b_i $ and $ 1> t_1 > t_2>-1 $, such that 
\begin{equation*}
\begin{aligned}
f_1(t):=\frac{q-p^1}{r^2}= A(3t^2-1)-a_1-b_1t^2+ B\left( 3t+\frac{3t^2-1}{2}\ln \frac{1-t}{1+t}\right) \geq 0, \\
f_2(t):=\frac{p^2-q}{r^2}= a_2+b_2t^2-A(3t^2-1)-B\left( 3t+\frac{3t^2-1}{2}\ln \frac{1-t}{1+t}\right)\geq 0, \textrm{ for }  t_2 \leq t\leq t_1,
\end{aligned}
\end{equation*}
and 
\begin{equation}\label{f12prime}
f_1(t_1)=f^{\prime}_1(t_1)=0,~~f_2(t_2)=f^{\prime}_2(t_2)=0.
\end{equation}

\par 
Let us rewrite 
\begin{equation}
\begin{aligned}
f_1(t)= t^2(3A-b_1)-a_1-A+ B\left( 3t+\frac{3t^2-1}{2}\ln \frac{1-t}{1+t}\right), \textrm{ and } \\
f_2(t)= t^2(-3A+b_2)+a_2+A-B\left( 3t+\frac{3t^2-1}{2}\ln \frac{1-t}{1+t}\right).
\end{aligned}
\end{equation}

\par 
Denote by
\begin{equation} 
g(t):=3t+\frac{3t^2-1}{2}\ln \frac{1-t}{1+t}, \textrm{ for } t \in (-1,1),
\end{equation}
and observe that $ g $ is an odd function.
From \eqref{f12prime} we  obtain the following system;
\begin{equation*}
\begin{cases}
      &f_1(t_1)= t_1^2(3A-b_1)-a_1-A+ Bg(t_1)=0 \\
      &f_1^\prime(t_1)= 2t_1(3A-b_1)+ Bg{^\prime}(t_1)=0\\
      &f_2(t_2)=t_2^2(-3A+b_2)+a_2+A-Bg(t_2)=0 \\
       & f_2^\prime(t_2)= 2t_2(-3A+b_2)-Bg^\prime(t_2)=0.
       \end{cases}
\end{equation*}
In order to simplify the case further, we assume that the free boundary of the desired 
 double-cone solution is symmetric with respect to the $(x,y)$-plane. Hence  $ t_1=-t_2 =t_0 \neq 0$, and 
\begin{equation}\label{systemt0}
\begin{cases}
      & t_0^2(3A-b_1)-a_1-A+ Bg(t_0)=0 \\
      & 2t_0(3A-b_1)+ Bg{^\prime}(t_0)=0\\
      &t_0^2(-3A+b_2)+a_2+A+Bg(t_0)=0 \\
       & -2t_0(-3A+b_2)-Bg^\prime(t_0)=0,
       \end{cases}
\end{equation}
where we used that $ g$ is an odd function, and $g^\prime$ is an even function.
The system \eqref{systemt0} is equivalent to the following system
\begin{equation} \label{systemAB}
\begin{cases}
      & 2A+a_1+a_2=0  \\
      & 6A-b_1-b_2=0 \\
      & t_0^2(b_2-b_1)+a_2-a_1+2 Bg(t_0)=0 \\
      & t_0(b_2-b_1)+ Bg{^\prime}(t_0)=0
       \end{cases}
\end{equation}
Now let us take $ b_1=b_2=b$, then $ 3A-b_1=-3A+b_2=0$ and
\begin{equation*}
\begin{aligned}
f_1(t)= -a_1-A+ Bg(t)= \frac{a_2-a_1}{2}+Bg(t), ~ f_2(t)= a_2+A-Bg(t)=\frac{a_2-a_1}{2}-Bg(t).
\end{aligned}
\end{equation*}
On the other hand, the assumption $b_1=b_2$ implies that $ B g^\prime({t_0})=0$,
and  $ a_2-a_1 =- 2Bg(t_0)>0$, hence $ B \neq 0$, and $  g^\prime({t_0})=0$. 
It is easy to verify that $ t_0 $ is unique in the interval $(0,1)$, $ g(t_0) >0$, and $ g ^\prime (t ) \geq 0$, for $ -t_0 \leq t \leq t_0$. Hence $ g$ is a monotone increasing function 
in the interval $ (-t_0,t_0)$, and therefore
\begin{equation} \label{f1f2}
\begin{aligned}
f_1(t)= \frac{a_2-a_1}{2g(t_0)}\left( g(t_0)- {g(t)} \right) \geq 0 \textrm{ and } \\
f_2(t)= \frac{a_2-a_1}{2g(t_0)}\left(g(t_0)+ {g(t)} \right)=\frac{a_2-a_1}{2g(t_0)}\left(-g(-t_0)+ {g(t)} \right)\geq 0 .
\end{aligned}
\end{equation}
Now we can write an explicit example.

\begin{example} \label{3Dexample}
Let 
\begin{equation*}
p^1(r,\phi,\theta)= -r^2, ~ p^2(r,\phi, \theta)= r^2,
\end{equation*}
in spherical coordinates, and 
\begin{equation*}
q(r,\phi,\theta)=-\frac{r^2 g(\cos\theta)}{g(t_0)},
\end{equation*}
where $ g(t)=3t+\frac{3t^2-1}{2}\ln \frac{1-t}{1+t}$ is defined for $ t \in (-1,1)$, and $ 0<t_0<1$ is chosen so that $ g ^\prime(t_0)=0$.
Then 
\begin{equation}\label{doublecone3}
u=
\begin{cases}
     p^1 & \mbox{if } ~0 \leq   \theta \leq \arccos(t_0)\\
      q & \mbox{if } ~ \arccos(t_0) \leq \theta \leq \pi -\arccos(t_0)\\
       p^2 & \mbox{if }~\pi-\arccos(t_0) \leq \theta \leq \pi
       \end{cases}
\end{equation}
is a double-cone solution to the double obstacle problem with obstacles $ p^1, p^2$.
\end{example}

\begin{proof}
Let $p^1=-r^2$ and $p^2=r^2$, then  $ a_1=-a_2=1$ and $ b_1=b_2=0$ in \eqref{3Dpol}.
Hence we obtain from \eqref{systemAB} that $  A=0,~ B=-1/g(t_0)$, and  $ q=-r^2 \frac{g(\cos \theta)}{g(t_0)}$.
Denote by 
\begin{equation}
\mathcal{C}_1:= \{ (r,\phi,\theta):  0 \leq \theta \leq \arccos(t_0) \}
\textrm{ and }
\mathcal{C}_2:= \{ (r,\phi,\theta):  \pi-\arccos(t_0)  \leq \theta \leq \pi \}.
\end{equation}
Then $q(r,\phi,\theta)$ is a harmonic function in  $\mathbb{R}^3\setminus (\mathcal{C}_1 \cup  \mathcal{C}_2)$, satisfying \eqref{3Ddc} and the boundary conditions \eqref{3Dbc}.
Hence the function $u$, defined in \eqref{doublecone3}, is  a homogeneous global solution to the double obstacle problem with obstacles $ p^1, p^2$.
The coincidence sets $ \{u =p^1\}= \mathcal{C}_1$ and $ \{u =p^2\}=\mathcal{C}_2$ are cones with a common vertex, thus $u$ is a double-cone solution.
\end{proof}

\par
It follows from our classification of blow-up solutions in $\mathbb{R}^2$ and from Example \ref{3Dexample}, that in  $\mathbb{R}^3$ there are at least four types of blow-ups;
polynomial, halfspace, double-cone solutions, and solutions, for which the free boundary is a union of four halfplanes. 
The complete analysis of  homogeneous global solutions and the regularity of the free boundary  for the double obstacle problem in dimension $ n=3$ 
we leave for a future publication.


\addcontentsline{toc}{section}{\numberline{}References}

\bibliographystyle{plain}
\bibliography{doubleobstacle}

\end{document}